\documentclass[11pt]{article}
\usepackage[usenames]{color}
\usepackage{xcolor}
\usepackage{hyperref}
\usepackage[utf8]{inputenc}
\usepackage{amsmath,amsfonts,amscd,amssymb,amsthm,latexsym,eucal,graphics,amsbsy,romannum}
\usepackage{enumitem}
\usepackage{makecell}

\newcommand{\Out}{\operatorname{Out}}  

\providecommand{\bbN}{{\mathbb N}} 
\providecommand{\bbF}{{\mathbb F}} 

\providecommand{\Syl}{{\mathrm{Syl}}}

\providecommand{\PSp}{{\mathrm{PSp}}}
\providecommand{\N}{{\mathrm{N}}} 
\providecommand{\C}{{\mathrm{C}}}
\providecommand{\Z}{{\mathrm{Z}}}

\providecommand{\rmO}{{\mathrm{O}}}
\providecommand{\F}{{\mathrm{F}}} 

\providecommand{\teq}{{\ \trianglelefteq \ }}
\providecommand{\te}{{\ \triangleleft \ }}

\providecommand{\Aut}{{\operatorname{Aut}}}
\providecommand{\Out}{{\operatorname{Out}}} 

\providecommand{\gen}[1]{\langle #1 \rangle} 

\providecommand{\gen}[1]{\langle #1 \rangle} 
\providecommand{\gen}[1]{\langle #1 \rangle} 

\providecommand{\olb}{{\overline{b}}}
\providecommand{\olc}{{\overline{c}}}

\providecommand{\olf}{{\overline{f}}}
\providecommand{\olg}{{\overline{g}}}
\providecommand{\olh}{{\overline{h}}}

\providecommand{\olu}{{\overline{u}}}
\providecommand{\olv}{{\overline{v}}}

\providecommand{\olx}{{\overline{x}}}

\providecommand{\olz}{{\overline{z}}}

\providecommand{\olG}{{\overline{G}}}
\providecommand{\olH}{{\overline{H}}}

\providecommand{\olM}{{\overline{M}}}

\providecommand{\olX}{{\overline{X}}}

\providecommand{\wtf}{{\widetilde{f}}}
\providecommand{\wtg}{{\widetilde{g}}}
\providecommand{\wth}{{\widetilde{h}}}

\providecommand{\wtu}{{\widetilde{u}}}
\providecommand{\wtv}{{\widetilde{v}}}

\providecommand{\wtz}{{\widetilde{z}}}

\providecommand{\wtA}{{\widetilde{A}}}

\providecommand{\wtE}{{\widetilde{E}}}

\providecommand{\wtG}{{\widetilde{G}}}
\providecommand{\wtH}{{\widetilde{H}}}

\providecommand{\wtK}{{\widetilde{K}}}
\providecommand{\wtL}{{\widetilde{L}}}
\providecommand{\wtM}{{\widetilde{M}}}

\providecommand{\wtW}{{\widetilde{W}}}
\providecommand{\wtX}{{\widetilde{X}}}
\providecommand{\wtY}{{\widetilde{Y}}}

\providecommand{\whf}{{\widehat{f}}}
\providecommand{\whg}{{\widehat{g}}}
\providecommand{\whh}{{\widehat{h}}}

\providecommand{\whu}{{\widehat{u}}}
\providecommand{\whv}{{\widehat{v}}}

\providecommand{\whG}{{\widehat{G}}}
\providecommand{\whH}{{\widehat{H}}}

\providecommand{\whK}{{\widehat{K}}}
\providecommand{\whL}{{\widehat{L}}}
\providecommand{\whM}{{\widehat{M}}}
\providecommand{\whN}{{\widehat{N}}}

\providecommand{\whW}{{\widehat{W}}}
\providecommand{\whX}{{\widehat{X}}}

\providecommand{\calO}{\mathcal{O}}

\providecommand{\bbF}{\mathbb{F}}

\providecommand{\bbN}{\mathbb{N}}

\newtheorem{lem}{Lemma}

\newtheorem*{maintheorem}{Main Theorem}

\title{On $\{2,3,5\}$-groups with conjugacy classes of distinct sizes}
\author{
Wei Zhou\thanks{Sobolev Institute of Mathematics, Novosibirsk 630090, Russia.
Email: zhouwyx@outlook.com.}
\thanks{The author was supported by the China Scholarship Council (Grant No. 202510100044).}
\footnotemark[5]
,
Ilya Gorshkov\thanks{School of Mathematical Sciences, Hebei Key Laboratory of Computational Mathematics and Applications, Hebei Normal University, Shijiazhuang 050024, P. R. China; Novosibirsk State Technical University, Novosibirsk 630073 Russia. Email: ilygor8@gmail.com.}
\thanks{The author was supported by the grant of The Natural Science Foundation of Hebei Province (Project No. A2023205045).} 
\footnotemark[5]
}

\date{}

\begin{document}

\pagenumbering{arabic}

\maketitle

\footnotetext[5]{The work is supported by the Mathematical Center in Akademgorodok under the agreement No. 075-15-2025-348 with the Ministry of Science and Higher Education of the Russian Federation.}

\noindent\textbf{Abstract.} 
    A finite group $G$ is called an $ah$-group 
    if any two distinct conjugacy classes of $G$ have distinct sizes.
    In this paper, we show that if $G$ is an $ah$-group and $\pi(G) \subseteq \{2,3,5\}$, 
    where $\pi(G)$ denotes the set of prime divisors of $|G|$, then $G \cong S_3$.



\section*{Introduction}

    In this paper, all groups considered are finite. 
    The sizes of conjugacy classes carry essential information about the algebraic structure of a finite group. 
    For example, Burnside \cite{Burnside} proved that a simple group has no conjugacy class of prime power size, 
    and Itô \cite{Ito1953} showed that any group with only one possible conjugacy class size other than $1$ is nilpotent.

    A group in which distinct conjugacy classes have distinct sizes is called an anti-homogeneous group, or simply an $ah$-group. 
    The unique known non-trivial $ah$-group is the symmetric group $S_3$, whose conjugacy class sizes are exactly $1$, $2$, and $3$. 
    In 1973, Markel \cite{Markel1973} proposed the famous $S_3$-conjecture: any non-trivial finite $ah$-group is isomorphic to $S_3$.

    Many researchers have been interested in this conjecture, and it has been completely resolved for solvable groups. 
    Markel \cite{Markel1973} first verified the $S_3$-conjecture for all supersolvable groups. 
    Gilotti \cite{Gilotti1975} verified it for $\{2,3\}$-groups and $\{2,5\}$-groups, 
    and Ward \cite{Ward1970} verified it for $\{2,3\}$-groups 
    or solvable groups containing a self-centralizing element of order $p$ for some prime $p$. 
    The conjecture for all solvable groups was independently verified 
    by Zhang \cite{zhang1994} 
    and by Knörr, Lempken, and Thielcke \cite{Knorr1995}.

    However, the general non-solvable cases remain open. 
    Partial progress was made by Arad, Muzychuk, and Oliver \cite{Arad2004B}, 
    who proved that the non-abelian socle of an $ah$-group must be isomorphic to specific groups. 
    They also showed that if $G$ is an $ah$-group with $\pi(G) \not\subseteq \{2,3,5,7\}$, 
    then the Fitting subgroup $\F(G) > 1$.
    Luo and Liu \cite{LuoLiu2019} showed that every finite non-abelian simple group is not an $ah$-group.

    In this paper, we extend the investigation of the $S_3$-conjecture to some non-solvable groups. 
    Our research focuses on $\{2, 3, 5\}$-groups. 
    We prove the following main result, 
    which contributes towards a full proof of the $S_3$-conjecture.

\begin{maintheorem}
If $G$ is an $ah$-group and $\pi(G)\subseteq\{2,3,5\}$, then $G\cong S_3$.
\end{maintheorem}

\section{Preliminaries}

    For $x\in G$, we write $x^G$ for the conjugacy class of $x$ in $G$.
    We denote by $\pi(G)$ the set of prime divisors of $|G|$, by $\omega(G)$ the set of element orders of $G$,
    and by $\N(G)$ the set of conjugacy class sizes of $G$ (also commonly denoted by $\operatorname{cs}(G)$).
    When a quotient group is denoted with a bar, hat, or tilde, the same decoration is used
    for the images of elements and subsets in that quotient.

\begin{lem}[\cite{Isaacs1976}, Corollary 2.24]
\label{Centre}
    Let $G$ be a finite group, $g\in G$, and $N$ be a normal subgroup of $G$.
    We have $|\C_{G/N}(Ng)|\leq |\C_G(g)|$.
\end{lem}

\begin{lem}[\cite{Gorenstein}, Theorem 5.2.3]
\label{Go}
    Let $A$ be a $p'$-group of automorphisms of the abelian group $P$. 
    We have $P=\C_P(A)\times[P,A]$.
\end{lem}

\begin{lem}[\cite{Mazurov1997}, Lemma 1]
\label{Frob}
    Let $G$ be a finite group, $K$ be a normal subgroup of $G$,
    and let $G/K$ be a Frobenius group with kernel $F$ and cyclic complement $C$.
    If $(|F|,|K|)=1$ and $F$ does not lie in $K\C_G(K)/K$, 
    then $r|C|\in \omega (G)$ for some prime divisor $r$ of $|K|$.
\end{lem}

\begin{lem}[\cite{Gas1952}, Theorem 1]
\label{complement}
    Let $N$ be an abelian normal subgroup of a finite group $G$. 
    If $N\leq H\leq G$, $N$ has a complement in $H$, 
    and $\mathrm{gcd}(|N|, |G:H|)=1$, then $N$ has a complement in $G$.
\end{lem}

    Recall that a finite group $G$ is called rational if, 
    for every $g\in G$ and every integer $m$ coprime to $|g|$, 
    the elements $g$ and $g^m$ are conjugate in $G$.
    Let $\pi$ be a set of primes.
    Following \cite{Knorr1995}, 
    we say that $G$ is a $\pi$-quasirational group, or a $\pi$-$qr$ group for short, 
    if $\N_G(U)/\C_G(U)\cong \Aut(U)$ for every cyclic $\pi$-subgroup $U$ of $G$.
    Equivalently, $G$ is a $\pi$-$qr$ group if, 
    for every $\pi$-element $u\in G$ and every integer $m$ coprime to $|u|$, 
    the elements $u$ and $u^m$ are conjugate in $G$.
    Thus every rational group is a $\pi$-$qr$ group for every set of primes $\pi$.
    Finally, every $ah$-group is rational: indeed, if $g\in G$ and $(m,|g|)=1$, then $\gen{g}=\gen{g^m}$, and hence $\C_G(g)=\C_G(g^m)$.
    Therefore the conjugacy classes $g^G$ and $(g^m)^G$ have the same size, 
    so they coincide by the definition of an $ah$-group.

\begin{lem}[\cite{Knorr1995}, Lemma 2]
\label{CGU}
    Let $G$ be a $\pi$-$qr$ group and $q\in \pi$.
    If $u\in G$ is a $q$-element,
    then $\C_{G}(u)$ is a $\pi_1$-qr group,
    where $\pi_1 = \pi \setminus \{q\}$.
\end{lem}

\begin{lem}[\cite{Knorr1995}, Lemma 5]
\label{25qr}
    Let $G$ be a $\{2, 5\}$-qr group and $g\in G$.
    The following hold:

    (a) There is no chief factor of order $5$ in $G$.

    (b) If $G_5\neq 1$, then $G_2$ is non-abelian.
        In particular, $8\mid |G|$ and $|\C_G(g)|\neq 2$.

    (c) If $|\C_G(g)|=4$, 
        then either $G_2\cong V_4$ or $G_2$ is non-abelian of order $8$.
\end{lem}

\begin{lem}[\cite{Feit1988}, Theorem B; \cite{Thompson2008}, Theorem 1.1]
\label{RF}
    If $H$ is a composition factor of a rational group, 
    then $H$ is isomorphic to one of the following:

    \begin{enumerate}[label=(\arabic*), noitemsep]

    \item A cyclic group of order $2, 3, 5, 7, 11$.
    \item An alternating group $A_n$, for $n\geq 5$.
    \item $\PSp_4(3)$.
    \item $\mathrm{Sp}_6(2)$.
    \item $\mathrm{O}^{+}_{8}(2)'$.
    \item $\mathrm{PSL}_3(4)$.
    \item $\mathrm{PSU}_4(3)$.
    
    \end{enumerate}
\end{lem}

\begin{lem}[\cite{Arad2004B}, Proposition 2.9]
\label{Factor}
    Any factor group of a rational group is also rational.
\end{lem}

\begin{lem}[\cite{Mazurov2006}, Lemma 15]
\label{Mazurov06}
    Let  $q$ and $p$ be distinct primes, 
    and let $H \langle \varphi \rangle$ be the semidirect product of a normal 
    $\{2, q, p\}'$-subgroup \( H \) and a cyclic group $\langle \varphi \rangle$ of order $q$. 
    Further, let  $H = [H, \varphi] \neq 1$, 
    and suppose the group $H \langle \varphi \rangle$ acts faithfully 
    on a vector space $V$ over the field $\mathbb{F}_p$ of order  $p$. 
    We have $\C_V(\varphi) \neq 0$.
\end{lem}

\begin{lem}[{\cite[Theorem 7.1]{Gorenstein}}]
\label{gore71}
    Let $G = H \times K$ and let $F$ be a splitting field for both $H$ and $K$. If $V/F$ and $W/F$ are irreducible $H$- and $K$-modules, respectively, then the product module $V \otimes_F W$ is an irreducible $G$-module. Conversely, every irreducible $G$-module over $F$ is equivalent to a product module of this form.
\end{lem}


\begin{lem}
\label{Sylp}
Let $G$ be a finite group and $g\in G$.
If $|\C_G(g)|=p$ where $p$ is a prime, then $|G|_p=p$.
\end{lem}

\begin{proof}
Let $g\in G$ be such that $|\C_G(g)|=p$. 
We have $g\in \C_G(g)$. 
Therefore $|g|=p$. 
Let $P$ be the Sylow $p$-subgroup of $G$ which contains $g$.
We know that $\Z(P)\leq \C_G(g)$.
Therefore $\Z(P)=\gen{g}$, which implies that $P=\C_P(g)=\gen{g}$.
Hence $|G|_p=p$.
\end{proof}

\begin{lem}
\label{trivial}
    Let $G$ be a finite group, $K \te G$, and $p$ be a prime. 
    Set $\olG=G/K$.
    Suppose $g \in G$ is a $p$-element such that $|\C_G(g)|_p = |g| = |\olg|$. 
    We have $|K|_p = 1$.
\end{lem}

\begin{proof}
    Suppose that $|K|_p > 1$.
    Since $|g|=|\olg|$, we have $\gen{g}\cap K=1$.
    Let $P$ be a Sylow $p$-subgroup of $K\gen{g}$ containing $g$.
    Since $K \te K\gen{g}$, $P\cap K$ is a Sylow $p$-subgroup of $K$. 
    Thus, $1 < P\cap K \teq P$.
    Hence $\Z(P)\cap K > 1$.
    Since $\gen{g}\cap K=1$,  $|\C_G(g)|_p \geq |\Z(P)\cap K||g|>|g|$, a contradiction.
    Hence $|K|_p=1$.
\end{proof}

\begin{lem}
\label{Orbit}
    Let $G$ be a finite group and $N$ be a normal subgroup of $G$. 
    Write $\alpha$ for the natural homomorphism from $G$ onto $G/N$. 
    For an element $x \in G$, let $\overline{x} = \alpha(x)$ and let $\overline{X}$ be the conjugacy class of $\overline{x}$ in $G/N$.
    Suppose that the preimage $\alpha^{-1}(\overline{X})$ partitions into $k$ distinct $G$-conjugacy classes 
    with representatives $x_1, x_2, \dots, x_k$. We have
$$
\sum_{i=1}^{k} \dfrac{1}{|\C_G(x_i)|} = \dfrac{1}{|\C_{G/N}(\overline{x})|}.
$$
\end{lem}

\begin{proof}
    We have $|\olX|=\dfrac{|G/N|}{|\C_{G/N}(\olx)|}$.
    Since $\alpha$ is a homomorphism with kernel $N$,
    $|\alpha^{-1}(\olX)|=|N||\olX|=|N|\dfrac{|G/N|}{|\C_{G/N}(\olx)|}=\dfrac{|G|}{|\C_{G/N}(\olx)|}$.
    On the other hand, $|\alpha^{-1}(\olX)|=\sum_{i=1}^{k}|x_i^G|=\sum_{i=1}^{k}\dfrac{|G|}{|\C_{G}(x_i)|}$.
    Hence $\sum_{i=1}^{k} \dfrac{1}{|\C_G(x_i)|} = \dfrac{1}{|\C_{G/N}(\overline{x})|}$.
\end{proof}

\begin{lem}
\label{Centre2}
    Let $G$ be a finite group and $N=N_1\times \cdots \times N_k$ be a minimal normal subgroup of $G$, 
    where each $N_i$ is a non-abelian simple group.
    Let $\calO=\{N_1^{g^m}\mid m\in \bbN\}$ be the orbit of $N_1$ under the action of $g$,
    and let $K=\bigoplus_{H\in\calO}H$ be the direct product of the subgroups in the orbit.
    If $|\calO|=n$, then $\C_K(g)$ contains a subgroup isomorphic to $\C_{N_1}(g^n)$.
    In particular, if $|\calO|=|g|$, then $\C_K(g)$ contains a subgroup isomorphic to $N_1$.
\end{lem}

\begin{proof}
    Let $M = \{(x, x^g,\ldots, x^{g^{n-1}}) \mid x\in \C_{N_1}(g^n) \}$.
    It is easy to check that $M$ is a subgroup of $K$ and $M\cong \C_{N_1}(g^n)$.
    Since $(x, x^g,\ldots, x^{g^{n-1}})^g=(x^{g^n},x^g,\ldots, x^{g^{n-1}})
    =(x, x^g,\ldots, x^{g^{n-1}})$,
    we have $M\leq \C_K(g)$.
\end{proof}

    For convenience, we write $[n,m]$ to denote the SmallGroup library identifier $\texttt{SmallGroup}(n,m)$.
    The computations in this paper were performed with GAP \cite{GAP}.
    We include in \cite{Git} the GAP code for some of the more involved computations.

\begin{lem}
\label{KS5}
    Let $G$ be a rational $\{2,3,5\}$-group, and $V$ be a normal subgroup of $G$
    which is an elementary abelian $p$-group.
    If $V$ is an irreducible $G/V$-module and $G/V\cong S_5$, then $p=2$.
    Moreover, in this case $G$ must be isomorphic to one of the following groups:
    $C_2\times S_5$, $[1920,240993]$, or $[1920,240996]$.
\end{lem}

\begin{proof}
    If $p=5$, by the Brauer character table of $S_5$,
    elements of order $3$ have non-trivial fixed points in $V$.
    Thus there exists an element of order $15$ in $G$, denoted by $x$.
    Since $G$ is a rational group,
    $\Aut(\gen{x})=C_2\times C_4$ is contained in $G$.
    But the Sylow $2$-subgroup of $S_5$ is not isomorphic to $C_2\times C_4$.
    Hence $p\neq 5$.

    If $p=3$, let $g\in G$ be an element of order $5$.
    If $\C_V(g)>1$, then we can obtain a contradiction as above.
    Hence $\C_V(g)=1$.
    By the Brauer character table of $S_5$, in this case $|V|=3^4$.
    Let $s$ be an element of order $4$ in $G$.
    Thus $sV$ is also an element of order $4$ in $G/V$.
    Since $G/V\cong S_5$, $\C_{G/V}(sV) = \gen{sV}$.
    By the Brauer character table of $S_5$,
    we have $\C_V(s)>1$.
    Let $v\in \C_V(s)$ be an element of order $3$.
    Hence $|sv|=12$.
    Since $G$ is a rational group,
    there exists $a\in G$ such that $(sv)^a = (sv)^5$.
    It follows that $s^a = s$ and $v^a = v^{-1}$.
    Hence $a\in \C_G(s)$ and so $aV \in \C_{G/V}(sV) = \gen{sV}$.
    Thus $a\in \gen{s}V$.
    Since $s\in \C_G(v)$ and $V$ is abelian,
    we have $a\in \C_G(v)$, which contradicts that $v^a = v^{-1}$.
    Therefore  $p\neq 3$ and $p$ must be $2$.

    Since $p=2$, $V$ is an irreducible $2$-module for $S_5$ over $\bbF_2$.
    The irreducible $2$-modules of $S_5$ are either $1$-dimensional or $4$-dimensional,
    hence $|G|$ is either $240$ or $1920$.
    By iterating through the SmallGroup library, 
    \cite[Code 2]{Git} identifies that $G$ is isomorphic to one of the following groups:
    $C_2\times S_5$, $[1920,240993]$, or $[1920,240996]$.
\end{proof}

\begin{lem}
\label{2or16}
    Let $G$ be a finite group and $K\teq G$ be a non-trivial $2$-group 
    such that $G/K\cong S_5$. 
    If there exists an element $g \in G\setminus K$ such that $|\C_G(g)|< 8$, 
    then $G=K\rtimes S$, where $S\cong S_5$ and $g\in S$.
    Moreover, $G$ is not an $ah$-group.
\end{lem}

\begin{proof}
    We prove the first assertion by induction on $|K|$.
    Let $T\leq K$ such that $K/T$ is a chief factor of $G$.
    Set $\olG=G/K$ and $\wtG =G/T$.
    The subgroup $\wtK$ can be seen as an irreducible $\olG$-module.
    By Lemma \ref{Centre}, $|\C_{\wtG}(\wtg)|\leq |\C_G(g)|<8$.
    By \cite[Code 4]{Git}, $\wtG$ is isomorphic to either
    $[1920,240993]$ or $[1920,240996]$,
    and $\wtg$ is contained in a complement of $\wtK$ in $\wtG$,
    denoted by $\wtH$.
    It is clear that $\wtH\cong S_5$.
    
    Let $H$ be the preimage of $\wtH$ in $G$.
    If $T=1$, then $G=K\rtimes H$, $H\cong S_5$, and $g\in H$.
    If $T\neq 1$, then $H/T\cong S_5$, $g\in H\setminus T$, and
    $|\C_H(g)|\leq |\C_G(g)|<8$.
    By the induction hypothesis applied to $H$ and $T$, we have
    $H=T\rtimes S$ for some $S\cong S_5$ with $g\in S$.
    Since $G=KH$ and $K\cap S=1$, we obtain $G=K\rtimes S$.

    Now we show that $G$ is not an $ah$-group.
    Let $x\in G$ such that $|\olx|=4$.
    Since $\olG\cong S_5$, we have $|\C_{\olG}(\olx)|=4$.
    Let $\olX$ be the conjugacy class of $\olx$ in $\olG$.
    Suppose that the preimage of $\olX$ partitions into $k$ distinct
    $G$-conjugacy classes with representatives $x_1,\ldots,x_k$.
    By Lemma \ref{Orbit}, we have 
    \begin{equation}
        \sum_{i=1}^{k} \dfrac{1}{|\C_G(x_i)|} = \dfrac{1}{4}. \tag{$\ast$}
    \end{equation}
    We have $4\mid |x_i|$ and
    $\overline{\C_G(x_i)}\leq \C_{\olG}(\olx_i)$ for $i=1,\ldots,k$.
    Since $K$ is a $2$-group and $|\C_{\olG}(\olx_i)|=4$ for $i=1,\ldots,k$,
    it follows that $|\C_G(x_i)|$ is a power of $2$ divisible by $4$.
    If $k=1$, then $|\C_G(x_1)|=4$.
    Since $x_1$ is a $2$-element and $|x_1|=|\olx_1|=4$,
    Lemma \ref{trivial} implies that $K=1$, a contradiction.
    Thus $k\geq 2$.
    If $G$ is an $ah$-group, then $|\C_G(x_1)|,\ldots,|\C_G(x_k)|$ are pairwise distinct.
    This contradicts $(*)$.
    Hence $G$ is not an $ah$-group.
\end{proof}

  \begin{lem}
  \label{AutPSp43extension}
      Let $G$ be a $\{2,3,5\}$-group and let $K$ be a normal subgroup of
      $G$
      such that $G/K\cong \Aut(\PSp_4(3))$.
      If $G$ contains an element $g$ of order $9$ such that $|\C_G(g)|=9$,
      then $G$ is not an $ah$-group.
  \end{lem}

  \begin{proof}
      Set $\olG=G/K$.
      By Lemma \ref{Centre}, $|\C_{\olG}(\olg)|\leq |\C_G(g)|=9$.
      The centralizer data of $\Aut(\PSp_4(3))$ show that $|\olg|=9$ and
      $|\C_{\olG}(\olg)|=9$. Hence Lemma \ref{trivial} implies that $K$ is
      a
      $\{2,5\}$-group.

      Suppose that $G$ is an $ah$-group. Thus $G$ is rational. Since
      $\Aut(\PSp_4(3))$ is not an $ah$-group, we have $K>1$.
      Let $L$ be a largest normal subgroup of $G$ contained in $K$, and
      let $\widetilde G=G/L$ and $V=K/L$.
      The subgroup $V$ is an irreducible module for $\widetilde G/V\cong \Aut(\PSp_4(3))$. 
      Moreover, $\widetilde G$ is rational by Lemma \ref{Factor}, and
      $
          9=|\C_{\olG}(\olg)|
          \leq |\C_{\widetilde G}(gL)|
          \leq |\C_G(g)|=9.
     $
      Thus $\wtg$ acts fixed-point-freely on $V$.
      By \cite[Code 6]{Git}, $\widetilde G$ is isomorphic to one of the
      following
      two groups:
      \[
          G_1\cong C_2^6\rtimes \Aut(\PSp_4(3)),
          \qquad
          G_2\cong C_2^6.\Aut(\PSp_4(3)),
      \]
      where the second extension is non-split.

      We first assume that $\widetilde G\cong G_1$. For each conjugacy
      class of
      $G_1$, \cite[Code 6]{Git} gives the order $n$ of a representative
      and the
      order $m$ of its centralizer, listed as pairs $(n,m)$ and ordered by
      $m$:
      \[
      (9,9),(12,12),(8,16),(8,16),(10,20),(10,20),(20,20),
      (6,24),(12,24),(12,24),(24,24),\ldots .
      \]
      The same argument below applies to $G_2$; its corresponding list is
      also
      given in \cite[Code 6]{Git}.

      Let $\Omega=\{|\C_G(x)|\mid x\in G\}$.
      Since $G$ is an $ah$-group, the elements of $\Omega$ correspond
      bijectively
      to the conjugacy classes of $G$. Hence
      \[
          |G|=\sum_{m\in \Omega}\frac{|G|}{m}.
      \]

      By Lemma \ref{Centre} and the list above, we have $1,2,3,4,5,6,8\notin \Omega $.
      Clearly $9\in \Omega$. If $10\in\Omega$, choose $y\in G$ with
      $|\C_G(y)|=10$. Hence $|y|$ divides $10$, and hence $|\widetilde y|$
      divides
      $10$. The list above gives $|\C_{\widetilde G}(\widetilde y)|\geq
      20$,
      contradicting Lemma \ref{Centre}. Thus $10\notin\Omega$. Similarly,
      $15\notin\Omega$.

      Suppose that $18\in\Omega$, and choose $w\in G$ with $|\C_G(w)|=18$.
      Since $|\widetilde w|$ divides $18$, the list above forces
      $\widetilde w$ to have order $9$ and
      $|\C_{\widetilde G}(\widetilde w)|=9$. Thus $\widetilde w$ and
      $\widetilde g$ are conjugate in $\widetilde G$. Let
      $x_1,\ldots,x_r$ be representatives of the $G$-conjugacy classes
      contained in the preimage of this $\widetilde G$-class. By Lemma
      \ref{Orbit},
      \[
          \sum_{i=1}^{r}\frac{1}{|\C_G(x_i)|}
          =
          \frac{1}{|\C_{\widetilde G}(\widetilde g)|}
          =
          \frac{1}{9}.
      \]
      Since the $G$-class of $g$ is one of these classes and
      $1/|\C_G(g)|=1/9$, the equality above shows that it is the only
      $G$-class above this class of $\widetilde G$. Hence $w$ and $g$ are
      conjugate in $G$, a contradiction to $|\C_G(w)|=18$.
      Therefore $18\notin\Omega$.
      Now
      \[
      \begin{aligned}
          |G|
          &\leq
          [
          \sum_{d\mid |G|}\frac1d
          -
          (
          1+\frac12+\frac13+\frac14+\frac15+\frac16+\frac18
          +\frac1{10}+\frac1{15}+\frac1{18}
          )]|G| \\
          &<
          (
          \frac{15}{4}-\frac{1007}{360}
          )|G|
          =
          \frac{343}{360}|G|
          <|G|,
      \end{aligned}
      \]
      a contradiction. Thus $\widetilde G\not\cong G_1$. The case
      $\widetilde G\cong G_2$ is identical, using the corresponding data
      in
      \cite[Code 6]{Git}. Hence $G$ is not an $ah$-group.
  \end{proof}

\begin{lem}
\label{real3auta6}
    Let $G$ be a finite group and $N$ be a normal subgroup of $G$.
    If $|N|=3$ and $G/N\cong S_6$ or $\Aut(A_6)$, then $G$ is not rational.
\end{lem}

\begin{proof}
    Suppose that the assertion of this lemma is false.
    Let $N=\gen{a}$, where $a$ is an element of order $3$,
    and set $\olG=G/N$.
    Since $G$ is rational, we have $G/\C_G(a)\cong \Aut(C_3)\cong C_2$.
    Let $H =\C_G(a)$. Note that $|\olG:\olH|=2$.

    Assume that $\olG \cong S_6$.
    Hence $\olH \cong A_6$.
    Let $b\in H$ such that $|b|=5$.
    Since the centralizer of a $5$-cycle in $S_6$ is generated by the cycle itself,
    we have $\C_{\olG}(\olb)=\gen{\olb}\leq \olH$.
    Since $(|N|,|b|)=1$, we have $\C_{\olG}(\olb)=\C_G(b)N/N$.
    Hence $\C_G(b)\leq H$.
    Since $G$ is rational, 
    there exists an element $x\in G$ such that $(ab)^x=(ab)^{11}$.
    It follows that $a^x=a^2$ and $b^x=b$.
    The latter implies $x\in \C_G(b)\leq H$, so $a^x=a$, a contradiction.

    Therefore, we must have $\olG \cong \Aut(A_6)$.
    If $\olH \cong S_6$, 
    choose $b\in H$ such that $|b|=2$ and $\olb$ is an odd permutation in $\olH$.
    Thus $\C_{\olG}(\olb)\leq \olH$.
    Similarly, we have $\C_G(b)\leq H$.
    Since $G$ is rational, 
    there exists an element $x$ in $G$ such that $(ab)^x=(ab)^5$.
    We have $a^x=a^2$ and $b^x=b$, a contradiction.
    
    Hence $\olH$ is not isomorphic to $S_6$.
    Hence $\olH$ is isomorphic to either $M_{10}$ or $PGL_2(9)$.
    In these cases, there exists an element $c$ in $H$ such that $|c|=8$.
    Since every element of order $8$ in $\Aut(A_6)$ is self-centralizing,
    we have $\C_{\olG}(\olc)=\gen{\olc}\leq \olH$.
    Similarly, we have $\C_G(c)\leq H$.
    Since $G$ is rational, 
    there exists an element $y\in G$ such that $(ac)^y = (ac)^{17}$.
    It follows that $a^y=a^2$ and $c^y=c$, a contradiction.
    
    Hence $G$ is not rational.
\end{proof}



\begin{lem}
\label{C2S5}
    Let $G$ be a rational group and $H$ be a normal $3$-subgroup of $G$ 
    such that $G/H \cong C_2 \times S_5$. 
    If there exists an element $g$ in $G$ such that $|\C_G(g)|=8$, 
    then $G \cong (H \rtimes C) \times D$ and $g\in C\times D$,
    where $H$ is an elementary abelian $3$-group, $C\cong C_2$, $D\cong S_5$ and $\C_H(C)=1$.
\end{lem}

\begin{proof}
    We prove it by induction on $|H|$.

    If $|H|=3$, \cite[Code 17]{Git} shows that $G\cong S_3\times S_5$.

    Suppose that $|H|>3$. 
    Let $V\leq \Z(H)$ be a minimal normal subgroup of $G$.
    Set $\olG = G/H$ and $\wtG = G/V$.
    Since $|\C_G(g)|=8$, $g$ is a $2$-element. 
    By Lemma \ref{Centre}, $|\C_{\olG}(\olg)|\leq 8$.
    It is easy to verify that $|\olg|=4$ and so $|g|=4$.
    Since $(|g|,|V|)=1$, we have $|\C_{\wtG}(\wtg)| = |\C_G(g)| = 8$. 
    By Lemma \ref{Factor}, $\wtG$ is rational.
    By induction, we have $\wtG\cong (\wtH \rtimes \wtX)\times \wtY$ and $\wtg \in \wtX\times \wtY$,
    where $\wtH=C_3^n$, $\wtX\cong C_2$ and $\wtY\cong S_5$.
    We can view $V$ as an irreducible module of $\wtX\times \wtY$. 
    Since $\C_V(g)=1$, $\wtg$ must act on $V$ without non-trivial fixed points. 
    By \cite[Code 17]{Git}, $V$ must be $1$-dimensional, which implies $|V|=3$.

    Since $G$ is rational and $|V|=3$, we have $\N_G(V)/\C_G(V)\cong \Aut(V)\cong C_2$.
    Hence $\overline{\C_G(V)}$ has index $2$ in $\olG$, which implies $\overline{\C_G(V)}\cong S_5$ or $C_2\times A_5$.
    Suppose that $\overline{\C_G(V)}\cong C_2\times A_5$. 
    Let $1\neq a\in V$ and $b\in \C_G(V)$ be an element of order $5$. 
    Hence $ab$ has order $15$. 
    Since $G$ is rational, there exists $x\in G$ such that $(ab)^x=(ab)^{11}=a^2b$. 
    Hence $a^x=a^2$ and $b^x=b$. The latter implies that $\olx\in \C_{\olG}(\olb)$. 
    However, in $C_2 \times S_5$, the centralizer of a $5$-element is entirely contained in $C_2 \times A_5 = \overline{\C_G(V)}$. 
    It follows that $x\in \C_G(V)$.
    Hence $x$ centralizes $V$, which contradicts $a^x=a^2$. 
    Thus $\overline{\C_G(V)}\cong S_5$.

    Let $z$ be an involution of $G$ such that $\gen{\olz}=\Z(\olG)$.
    Thus $\olG =\overline{\C_G(V)} \times \gen{\olz}$. 
    Since $z \notin \C_G(V)$, $a^z=a^{-1}$ for any $a\in V$.
    Recall that $\C_V(g)=1$, so we also have $a^g=a^{-1}$. Let $\sigma = gz^{-1}$. 
    Hence $\sigma \in \C_G(V)$ and $\sigma$ has order $4$.
    Without loss of generality, we can let $\wtX= \gen{\wtz}$.
    since $\wtg\in \wtX\rtimes \wtY$, 
    $\widetilde{\sigma}\in \wtY$ and $\overline{\sigma}$ can be seen as a $4$-cycle.
    It follows that $\wtY$ also acts trivially on $V$.
    Since the Schur multiplier of $S_5$ is $C_2$, 
    the extension $V.\wtY$ splits. 
    Let $V.\wtY = V\times D$.
    Thus $D\cong S_5$ and $\sigma \in D$.

    Let $h\in H$. We have $h^\sigma = hb$, where $b\in V$.
    If $b\neq 1$, then $h^{\sigma^3}=hb^3=h$.
    Since $(|\sigma|,3)=1$, we have $h^{\sigma}=h$.
    It follows that $D$ centralizes $H$.
    Since $V.\wtX$ is normal in $V.(\wtX\times \wtY)=V\gen{z}D$
    and $\sigma$ centralizes $V\gen{z}$, 
    we have $D$ also centralizes $V\gen{z}$.
    Therefore $D$ centralizes $z$.
    Let $C=\gen{z}$. 
    Therefore, $G \cong (H\rtimes C)\times D$ and $g = \sigma z \in D \times C$.

    Finally, we prove that $\C_H(C)=1$ and that$H$ is elementary abelian.
    Let $T=H\rtimes C$. Thus $T\cong G/D$ is also rational.
    Suppose that $\C_H(C)>1$. Let $u\in \C_H(C)$ be an element of order $3$.
    Hence $\N_T(\gen{uz})/\C_T(\gen{uz})\cong \Aut(\gen{uz})\cong C_2$.
    It follows that $|\N_T(\gen{uz})|_2 = 4$, which contradicts that $|T|_2=2$.
    Hence $\C_H(C)=1$, which implies that $H$ is abelian.
    Suppose that there exists $v\in H$ such that $v$ is of order $9$.
    Thus $|\N_G(\gen{v})/\C_G(\gen{v})|\leq |T/H|=2$.
    However, $\Aut(\gen{v})\cong C_6$.
    This contradiction shows that $H$ must be elementary abelian.
\end{proof}

\begin{lem}
\label{S3S5}
    Let $M$ be a group and $V$ a minimal normal subgroup of $M$.
    Suppose that $V$ is a $p$-group for some $p\in\{2,5\}$ 
    and that $M/V\cong (H\rtimes C)\times D$,
    where $H$ is an elementary abelian $3$-group, $C\cong C_2$, $D\cong S_5$ and $\C_H(C)=1$.
    If $M$ contains an element whose centralizer has order $8$, 
    then $M$ is not a rational group.
\end{lem}

\begin{proof}
    Let $N = H \rtimes C$. 
    We can view the minimal normal subgroup $V$ as an irreducible $p$-module of $N\times D$. 
    By Lemma \ref{gore71}, $V$ can be decomposed as $V \cong V_1 \otimes V_2$, 
    where $V_1$ is an irreducible $N$-module and $V_2$ is an irreducible $D$-module.

	    First, we consider the case where $p=5$.
    
    We write the group operation of $V$ additively, 
    and we show that there exists an element $w \in V$ such that $g \cdot w \neq 2w$ for all $g \in M$. 
    Since $V$ is abelian, 
    conjugation by $g \in M$ on $V$ is determined by the action of its image in $M/V \cong N \times D$.

    By \cite[Code 18]{Git}, for any irreducible $5$-module $V_2$ of $S_5$,
    there exists an element $v \in V_2$ such that $x \cdot v \notin \{2v, 3v\}$ for all $x \in D$. 
    Let $w = u \otimes v \in V$, where $u \in V_1 \setminus \{0\}$.

    Suppose that $w$ is conjugate to $2w$ in $M$. 
    Hence there exists $(y, x) \in N \times D$ such that $(y \cdot u) \otimes (x \cdot v) = 2(u \otimes v)$. 
    This implies that $y \cdot u = \lambda u$ and $x \cdot v = \mu v$ 
    for some $\lambda, \mu \in \mathbb{F}_5^\times$ satisfying $\lambda \mu \equiv 2 \pmod 5$. 
    The possible pairs for $(\lambda, \mu)$ over $\mathbb{F}_5$ are $(1, 2), (2, 1), (3, 4)$, and $(4, 3)$. 
    From the choice of $v$, we have $\mu \notin \{2, 3\}$. 
    Thus, $\lambda \in \{2, 3\}$. 
    It follows that the order of $y$ must be divisible by 4, which contradicts that $|N|_2 = 2$. 
    Therefore, $w$ is not conjugate to $2w$ in $M$. Hence, $M$ is not a rational group.

    Next, we consider the case where $p=2$.

	    The dimension of $V_1$ must be $1$ or $2$.
    Let $K$ be the preimage of $H$ in $M$, that is, $K/V\cong H$.
    Suppose that $\dim(V_1)=1$. 
    Thus $H$ acts trivially on $V_1$, and so $H$ acts trivially on $V$.
    Hence $K$ is nilpotent.
    It follows that the Sylow $3$-subgroup of $K$, denoted by $T$,
    is normal in $M$.
    Hence $M/T\cong V.(C\times D)$ is rational and contains an element whose centralizer has order $8$.
    By \cite[Code 18]{Git}, it is impossible.
    Hence $\dim(V_1)=2$.
    In this case, $H/\C_{H}(V_1)\cong C_3$.
    We have $K/\C_K(V)\cong H/\C_{H}(V)\cong C_3$.
	    Let $L\in \Syl_3(\C_K(V))$.
    Since $\C_{K}(V)$ must be nilpotent, $L$ is normal in $M$.
    Hence $M/L\cong V.(S_3\times S_5)$ is also rational and contains an element whose centralizer has order $8$.
	    By \cite[Code 18]{Git}, this is also impossible.
    Hence this lemma holds.
\end{proof}

\section{Proof of Theorem}

    Throughout this section, 
    we assume that $G$ is a non-solvable $ah$-group with $\pi(G) \subseteq \{2,3,5\}$,
    and let $K$ be the largest normal subgroup of $G$ such that $G/K$ is non-solvable. The quotient $G/K$ contains a unique minimal normal subgroup 
    $R \cong R_1 \times \cdots \times R_k$, 
    where, by Lemma \ref{RF}, each $R_i$ is isomorphic to one of
    $A_5$, $A_6$, or $\PSp_4(3)$. 
    Set $\olG=G/K$.
    We know that $G$ is rational. 
    By Lemma \ref{Factor}, $\olG$ is also rational.

\begin{lem}
    For any element $g\in G$, $|\C_G(g)|\neq 2$. 
\end{lem}

\begin{proof}
     Suppose that there exists $g \in G$ such that $|\C_G(g)| = 2$.
     By Lemma \ref{Sylp}, $|G|_2 =2 $.
     Since $G$ is non-solvable, it has a non-abelian simple composition factor. 
     The order of every non-abelian simple group is divisible by $4$, 
     and hence $|G|_2\geq 4$, a contradiction.
     Therefore $|\C_G(g)|\neq 2$.
\end{proof}

\begin{lem}
    For any element $g\in G$, $|\C_G(g)|\neq 3$. 
\end{lem}

\begin{proof}
    Suppose that there exists $g \in G$ such that $|\C_G(g)| = 3$. 
    By Lemma \ref{Sylp}, $|G|_3 = 3$. 
    Therefore, $G$ has a unique composition factor, which is isomorphic to $A_5$. 
    Since $\olG$ is rational, $\olG \cong S_5$.
    It follows that $|\C_{\olG}(\olg)|=6$, which contradicts Lemma \ref{Centre}.
    Hence $|\C_G(g)|\neq 3$.
\end{proof}

\begin{lem}
    For any element $g \in G$, $|\C_G(g)| \neq 4$.
\end{lem}

\begin{proof}

    Suppose that there exists $g \in G$ such that $|\C_G(g)| = 4$.
    By Lemma \ref{Centre}, $g\notin K$.
    It follows that $|\olg|=2$ or $4$.
    
    We first consider the case $|\olg|=2$.
    If $R_i^{\olg}\neq R_i$ for some $i$, then by Lemma \ref{Centre2},
    $\C_{\olG}(\olg)$ contains a subgroup isomorphic to $R_i$,
    which contradicts $|\C_{\olG}(\olg)|\leq 4$.
    Hence $R_i^{\olg}=R_i$ for $i=1,\ldots,k$.
    By \cite[Code 1]{Git}, this forces $R\cong A_5$.
    Since $\olG$ is rational, $\olG\cong S_5$.
    However, in $S_5$, the centralizer of any involution has order at least $6$,
    contradicting $|\C_{\olG}(\olg)|\leq 4$.
    Therefore $|\olg|=4$.

    Now $|\olg|=4$.
    If $R_i^{\olg^2}\neq R_i$ for some $i$, then similarly Lemma \ref{Centre2}
    shows that $\C_{\olG}(\olg)$ contains a subgroup isomorphic to $R_i$,
    a contradiction.
    Thus $R_i^{\olg^2}=R_i$ for $i=1,\ldots,k$.
    If $R_i^{\olg}\neq R_i$ but $R_i^{\olg^2}=R_i$, then $\olg\notin R$,
    which implies $|\C_R(\olg)|\leq 2$.
    By \cite[Code 1]{Git}, this is impossible.
    Hence $R_i^{\olg}=R_i$ for $i=1,\ldots,k$.
    By \cite[Code 1]{Git}, it follows that $R$ is isomorphic to either
    $A_5$ or $A_6$.
    If $R\cong A_6$, then $\olG\cong S_6$ or $\Aut(A_6)$,
    neither of which contains an element with centralizer of order $4$.
    Therefore $R\cong A_5$ and $\olG\cong S_5$.

    By Lemma \ref{trivial}, $|K|$ must be a $\{3,5\}$-group.
    Thus $K$ is solvable.
    Let $T\te K$ such that $K/T$ is a chief factor of $G$.
    If $K>1$, then $K/T$ can be seen as an irreducible $3$- or $5$-module of $\olG$.
    By Lemma \ref{KS5}, $G/T$ must be isomorphic to one of the following groups:
    $C_2\times S_5$, $[1920, 240993]$, or $[1920, 240996]$.
    In all cases, $\C_{G/T}(gT)\neq 4$.
    Hence $K=1$ and so $G$ is isomorphic to $S_5$, which is not an $ah$-group.
    This contradiction completes this proof.
\end{proof}

\begin{lem}
    For any element $g \in G$, $|\C_G(g)| \neq 9$.
\end{lem} 

\begin{proof}
    Suppose that this is false and there exists $g\in G$ such that $|\C_G(g)|=9$.
    Hence $|g|=3$ or $9$.
    By Lemma \ref{Centre}, $g\notin K$.

\medskip
{\sl (1) $|g|\neq 3$. Hence $|g|=9$.}

    Suppose that $|g|=3$.
    If $R_i^{\olg}\neq R_i$, then by Lemma \ref{Centre2},
    $\C_{\olG}(\olg)$ contains a subgroup isomorphic to $R_i$,
    which contradicts that $|\C_{\olG}(\olg)|\leq |\C_G(g)|=9$.
    Hence $R_i^{\olg}=R_i$.
	    Since the centralizer of any element of order $3$ in $\PSp_4(3)$ has size at least $54>9$,
	    $R_i$ is not isomorphic to $\PSp_4(3)$ for $1\leq i \leq k$.
    If $R_i\cong A_6$, then $|\C_{R_i}(g)|=9$. 
    This forces $k=1$, which implies that $\olG$ must be $S_6$ or $\Aut(A_6)$.
    However, in both $S_6$ and $\Aut(A_6)$, 
    elements of order $3$ have centralizers of order at least $18$.
    This contradiction shows that $R_1\cong A_5$ and $1\leq k\leq 2$.
    If $R\cong A_5\times A_5$, then $R<\olG \lesssim \Aut(R)$.
    By \cite[Code 3]{Git}, any rational subgroup of $\Aut(A_5 \times A_5)$ containing $R$ has no elements of order $3$ with centralizer size $9$.
    Hence $\olG \cong S_5$.

    Let $T$ be the largest normal subgroup of $G$ such that $K/T$ is divisible by two primes, and set $\whG=G/T$.
    Choose $H\leq K$ such that $\whH$ is the minimal normal subgroup of $\whG$,
    and set $\wtG=G/H$.
    Suppose that $\whK$ is not solvable. 
    Thus $\whH$ is a direct product of non-abelian simple groups.
    Similarly, we can get that $\whH \cong A_5$ and $\whg$ acts on $A_5$ as an inner automorphism. 
    Hence $|\C_{\whG}(\whg)|\geq |\C_{\whH}(\whg)||\whg|=18>9$, a contradiction.
    Hence $\whK$ is solvable. 
    It follows that $\wtK$ is a $p$-group and $\whH$ is a $t$-group.
    By Lemma \ref{KS5} and Lemma \ref{2or16}, 
    we have $p=2$ and $\wtG\cong \wtK \rtimes \wtM$ where $\wtM \cong S_5$ and $\wtg \in \wtM$.

    If $g$ acts on $\wtK$ with non-trivial fixed points, 
    then the size of the centralizer of $\wtg$ in $\wtG$ must exceed $6\times 2>9$,
    a contradiction.
    Hence $\wtK\rtimes \gen{\wtg}$ is a Frobenius group.
    It is clear that $\wtK$ acts on $\whH$ non-trivially.
    By Lemma \ref{Frob}, $\C_{\whH}(\whg)\neq 1$.
    If $t=5$, then $|\C_{\whG}(\whg)|\geq 3\times 5\geq 9$, a contradiction.
    Hence $t=3$. 
    Since $|g|=3$, we have $\whH\cap \gen{\whg}=1$.
    As $(|\whH|,|\whG: \whH\gen{\whg}|)=1$,
    by Lemma \ref{complement}, $\whM\cong \whH\rtimes S_5$, 
    which implies that $|\C_{\whG}(\whg)|\geq 6\times 3>9$, a contradiction.
    Hence $K$ is a $2$-group.
    It follows that $|G|_3=3$, which contradicts that $|\C_G(g)|=9$.
    Hence $g$ is not of order $3$ and must be of order $9$.

\medskip
{\sl (2) $|\olg|=3$ and $\olG\cong S_5$. }
    
    Suppose that $|\olg|=9$. 
    Assume that $R_i^{\olg}\neq R_i$ for some $i$.
    If $R_i^{\olg^3}\neq R_i$, then by Lemma \ref{Centre2},
    $\C_{\olG}(\olg)$ contains a subgroup isomorphic to $R_i$, a contradiction.
    Hence $R_i^{\olg^3}=R_i$, which implies that $\olg^3$ acts on $R_i$ as an inner automorphism.
    By Lemma \ref{Centre2}, $\C_{R}(\olg)$ contains a subgroup isomorphic to $\C_{R_i}(\olg^3)$.
    If $R_i\cong \PSp_4(3)$, then $|\C_{R_i}(\olg^3)|>9$, a contradiction.
    If $R_i\cong A_6$, then $|\C_{R_i}(\olg)|= 9$.
    Since $R$ does not contain an element of order $9$, 
    we have $\olg\notin R$ and so $|\C_{\olG}(\olg)|\geq 3\times 9 >9$, a contradiction.
    Hence $R_i\cong A_5$.
    This forces that $R\cong A_5\times A_5\times A_5$ 
    and $\olG$ is isomorphic to a subgroup of $\Aut(R)$.
    By \cite[Code 5]{Git}, 
	    it can be verified that such $\olG$ cannot be both rational and contain an element whose centralizer has size $9$.
    Hence $R_i^{\olg}=R_i$ for $i=1,\ldots,k$.
    If $R_i\cong \PSp_4(3)$, then $k=1$ and $\olG\cong \Aut(\PSp_4(3))$.
    By Lemma \ref{AutPSp43extension}, $G$ is not an $ah$-group.
    Hence $R_i$ is isomorphic to either $A_5$ or $A_6$.
    Therefore, $\olg^3$ must centralize $R_i$ and so centralize $R$,
    which implies that $\gen{\olg^3}$ is also a normal subgroup of $\olG$,
	    a contradiction.
    Hence $|\olg|=3$.
    By the same argument presented previously, we have $\olG\cong S_5$.

\medskip
{\sl (3) Final contradiction.}

    We use the same notation as in (1).
    If $|\whg|=3$, we arrive at the same contradiction as before.
    Hence $|\whg|=9$.
	If $\whH$ is not solvable, let $\whH = N_1\times \cdots \times N_n$,
    where $N_1\cong \cdots \cong N_n$ are non-abelian simple groups.
    If $H<K$, then by Lemma \ref{KS5}, $\wtK$ must be a $2$-group.
    Hence $\whg^3\in \whH$.
    By the previous argument, $\whH$ is isomorphic to either
    $A_5\times A_5\times A_5$ or $A_5$.
    Since $\whG/\C_{\whG}(\whH)\lesssim \Aut(\whH)$, 
    $\C_{\whG}(\whH)$ or a quotient of $\C_{\whG}(\whH)$ must contain a subgroup isomorphic to $A_5$.
    However, this implies that $\whG$ has no element of order $9$, a contradiction.
    Hence $\whK$ is solvable.
    Moreover, $\wtK$ is a $2$-group and $\whH$ is a $3$-group.

    Let $h\in G$ such that $\whh$ is of order $5$.
    Since $\wtg$ has no non-trivial fixed points on $\wtK$, 
    the $2$-modular character table of $S_5$ implies that $\wth$ acts on $\wtK$ without
    non-trivial fixed points.
    Hence $\wtK \gen{\wth}$ is a Frobenius group.
    By Lemma \ref{Frob}, there exists an element of order $15$ in $\whK\gen{\whh}$.
    Since $\whG$ is rational, there exists an element of order $2$ in $\C_{\whG}(\whh)$.
    However, $|\C_{\olG}(\olh)|=5$ and $\wth$ acts on $\wtK$ without non-trivial fixed points, which implies that $|\C_{\wtG}(\wth)|=5$.
    This contradiction completes the proof.
\end{proof}

\begin{lem}
    For any element $g \in G$, $|\C_G(g)| \neq 6$.
\end{lem}

\begin{proof}
    Assume this lemma is false and let $g\in G$ such that $|\C_G(g)|=6$.
    We have $|g|\in \{2,3,6\}$.

\medskip
{\sl (1) $|g|$=6.}
    
    Suppose that $|g|=2$.
    Since $G$ is non-solvable, $|G|_2\geq 4$.
    It follows that $|\C_G(g)|_2\geq 4$, which contradicts that $|\C_G(g)|=6$.
    Hence $|g|\neq 2$.
    
    Suppose that $|g|=3$.
    Let $h$ be an element of order $6$ in $\C_G(g)$.
    It is clear that $\C_G(g)$ is abelian and $h, g$ are in different conjugacy classes of $G$.
    We have $|\C_G(g)|\leq |\C_G(h)|\leq |\C_G(h^2)|=|\C_G(g)|$.
    It follows that $G$ contains two conjugacy classes with the same size, 
    a contradiction. 
    Hence $|g|=6$ and (1) holds.

    Now we let $g=uv$, where $|u|=2$ and $|v|=3$.

\medskip
{\sl (2) $|\overline{g}|\notin \{1, 2\}$.}
    
    Let $g=uv$, where $|u|=2$ and $|v|=3$.
    We have $|\C_{\overline{G}}(\overline{g})|\leq |\C_G(g)|=6$ and $|\overline{g}|\in \{1, 2, 3, 6\}$.

    If $|\overline{g}|=1$, then $|\C_G(g)|=|\overline{G}| > 6$, a contradiction. 
    Hence $|\overline{g}|\neq 1$.

    If $|\overline{g}|=2$, then $|\overline{G}|_2 \geq 4$. 
    It follows that $|\C_{\overline{G}}(\overline{g})|=4$.
    Since $\overline{G}$ is rational, $\overline{G}$ is a $\{2, 5\}$-qr group.
    By Lemma \ref{25qr}, $|\overline{G}|_2\leq 8$.
    Since $\overline{G}$ is rational and non-solvable, $\overline{G}$ must be isomorphic to $S_5$.
    However, the centralizer of an element of order $2$ in $S_5$ has order $8$ or $12$, which contradicts that $|\C_{\overline{G}}(\overline{g})|=4$.
    Hence $|\overline{g}|\neq 2$.

\medskip
{\sl (3) $|\overline{g}|\neq 3$. Hence $|\overline{g}|=6$.}

    If $|\overline{g}|=3$, then $|\overline{G}|_3 = 3$.
    Since $\overline{G}$ is rational and non-solvable, 
    $\overline{G}$ must be isomorphic to $S_5$.
    Therefore $|\C_{\overline{G}}(\overline{g})| = 6 = |\C_G(g)|$.

{\sl (3.1) Let $T$ be a subgroup of $G$ such that $K/T$ is a chief factor of $G$.
    The quotient $K/T$ is an abelian $2$-group.}
    
    Since $S_5$ is not an $ah$-group, $K>1$.
    We have $K/T \cong Q_1\times Q_2 \times \cdots \times Q_r$,
    where $Q_1\cong \cdots \cong Q_r$ are simple groups.
    Since $6 = |\C_{\overline{G}}(\overline{g})|\leq |\C_{G/T}(gT)| \leq |\C_G(g)|=6$,
    we have $|\C_{G/T}(gT)| = 6$.
    If $K/T$ is not abelian, then $Q_1\in \{A_5, A_6, \PSp_4(3)\}$.
    If there exists some $i$ such that $Q_i^{vT}\neq Q_i$, 
    then $\C_{G/T}(vT)$ contains a subgroup which is isomorphic to $Q_i$.
    Hence $|\C_{G/T}(vT)|_2 \geq 4$.
    Since $\C_{G/T}(gT)=\C_{G/T}(uT)\cap \C_{G/T}(vT) = \C_{\C_{G/T}(vT)}(uT)$,
    We have $|\C_{G/T}(gT)|_2\geq 4$, a contradiction.
    Hence $Q_i^{vT} = Q_i$ for $i=1,\ldots,r$.
    Since the outer automorphism groups of $A_5, \PSp_4(3)$ are $C_2$
    and the outer automorphism group of $A_6$ is $C_2\times C_2$, 
    the action of $vT$ on $Q_i$ is equivalent to an inner automorphism on $Q_i$.
    It follows that $|\C_{Q_i}(vT)|\geq 3$.
    If $|uT|=1$, then $|\C_{G/T}(gT)|=|\C_{G/T}(vT)|\geq 9$, 
    which contradicts that $|\C_{G/T}(gT)|=6$.
    Hence $|uT|=2$ and $|gT|=6$.
    Let $uT = u_1 u_2 \ldots u_r$ where $u_i\in Q_i$.
    If $Q_i \cong A_5$ or $A_6$, then $u_i \notin \C_{Q_i}(vT)$.
    It follows that $uT\notin \C_{K/T}(vT)$, a contradiction.
    So $Q_i \cong \PSp_4(3)$.
    Hence $|\C_{Q_i}(u_i)|_3\geq 3$ and so $|\C_{K/T}(uT)|_3\geq 3$.
    Since $vT\notin K/T$ and $vT\in \C_{G/T}(uT)$, 
    we have $|\C_{G/T}(uT)|_3 \geq 9$.
    It follows that $|\C_{G/T}(gT)|_3 \geq 9$, a contradiction.
    Thus $K/T$ is abelian.

    By Lemma \ref{KS5}, $K/T$ is a $2$-group. Hence (3.1) holds.

\medskip
{\sl (3.2) $K$ is a $2$-group.}

    Suppose that $K$ is not a $2$-group.
    Thus there exist subgroups $L,J\leq K$ such that $K/L$ is a $2$-group
    and $L/J$ is a chief factor of $G$ which is not a $2$-group.
    Let $\wtG=G/L$. 
    
    If $u\notin L$, then $|\wtu| \neq 1$.
    We have 
    $6 = |\C_{\olG}(\overline{v})|=|\C_{\wtG}(\wtv)\wtK / \wtK|
    = |\C_{\wtG}(\wtv)|/|\C_{\wtK}(\wtv)|$.
    Since $\wtu \in \C_{\wtK}(\wtv)$, 
    $|\C_{\wtK}(\wtv)|\geq 2$ and so $|\C_{\wtG}(\wtv)|_2\geq 4$.
    It follows that $|\C_{\wtG}(\wtg)|_2 \geq 4$, 
    which contradicts that $|\C_{\wtG}(\wtg)|\leq 6$.
    Hence $u\in L$.
    
    By the same argument as in Step (3.1), we deduce that $L/J$ is abelian.
    Let $L/J$ be a $q$-group where $q\in \{3,5\}$.
    Set $\whG=G/J$. 
    It is clear that $u\in J$.
    Hence $\whg=\whv$.
	    Since $\C_{\wtK}(\wtg)=1$ and $\C_{\whL}(\whg)=1$,
    we have $\C_{\whK}(\whg)=1$.
    It follows that $\gen{\whK, \wtg}$ is a Frobenius group and so $\whK$ is nilpotent.
    Let $P\in \Syl _2(\whK)$.
    Hence $P$ is normal in $\whG$.
    We have $\whG/P$ is also a rational group, which contradicts Lemma \ref{KS5}.
    Hence $K$ is a $2$-group.

\medskip
{\sl (3.3) Final contradiction for (3).}
    
    By Lemma \ref{2or16}, $G$ is not an $ah$-group.
    This contradiction shows that (3) holds.

\medskip
{\sl (4) $\overline{G}$ is isomorphic to one of the following groups:
$S_5$, $S_6$, and $\Aut(A_6)$.}

    Since $|\olg|=6$, $|\overline{u}|, |\overline{v}|\neq 1$.
    Recall that $R=R_1\times \cdots \times R_k$ is the unique minimal normal subgroup of $\olG$.
    If there exists some $R_i\in \{R_1,\ldots,R_k\}$ such that $R_i^{\overline{v}}\neq R_i$,
    then by Lemma \ref{Centre2},
    $\C_{\olG}(\olv)$ contains a subgroup which is isomorphic to $R_i$.
    It follows that $|\C_{\olG}(\olv)|_2\geq 4$ and so $|\C_{\olG}(\olg)|>6$, 
    a contradiction.
    Hence $R_i^{\olv}=R_i$ for $i=1,\ldots,k$.
    Since $R_i$ is isomorphic to one of $A_5$, $A_6$, or $\PSp_4(3)$, 
    we can see the action of $\olv$ on $R_i$ as an inner automorphism.
    Since $\C_{\olG}(R)=1$, 
    we have $\olv \in R$.
    It is clear that $\olu\notin R$.
    Since $\olG$ is a $qr$-group,
    we have $\N_{\olG}(\gen{\olv})/\C_{\olG}(\gen{\olv})\cong C_2$.
    As $|\N_R(\gen{\olv})|_2 
        = |\N_{R_1}(\gen{\olv})|_2 \ldots |\N_{R_k}(\gen{\olv})|_2\geq 2^k$,
    we have $|\C_{\olG}(\gen{\olv})|_2\geq 2^k$.
    If $k>1$, then $|\C_{\olG}(\olv)|_2 \geq 4$, 
    which implies that $|\C_{\olG}(g)|_2\geq 4$, a contradiction.
    Hence $k=1$.
    It is easy to verify that $\Aut(\PSp_4(3))$ contains no element whose centralizer has order $6$.
    Thus, $R$ is isomorphic to either $A_5$ or $A_6$, and so (4) holds.

\medskip
{\sl (5) $5\notin \pi(K)$. Thus, $K$ is a $\{2,3\}$-group and is solvable.}

    Suppose that $5\in \pi(K)$. 
    Let $T$ be a largest normal subgroup of $G$ such that $5$ divides $|K/T|$.
    Choose $H/T$ to be a minimal normal subgroup in $G/T$. 
    Thus $5$ divides $H/T$ and $K/H$ is a $\{2,3\}$-group.
    Set $\whG=G/T$.
    
    If $\whH$ is non-solvable, 
    as in (4), we can show that $\whH$ is non-abelian simple group 
    and $\whv$ acts on $\whH$ as an inner automorphism.
    It is clear that $\whH$ and $\C_{\whK}(\whH)$ are both characteristic subgroups of $\whK$.
    By the definition of $T$, $\C_{\whK}(\whH)=1$.
    Hence $K/H\cong \Out(\whH)$, which implies that $|K/H|\in \{1,2,4\}$.
    Since $\whG/(\whH \C_{\whG}(\whH))$ is a $2$-group, 
    we have $\whv\in \whH \C_{\whG}(\whH)$.
    Let $\whv = v_1 v_2$, where $1\neq v_1 \in \whH$ and $1\neq v_2\in \C_{\whG}(\whH)$.
    It is easy to see that $v_1, v_2\in \C_{\whG}(\whu)$.
    Hence $|\C_{\whG}(\whg)|_3\geq 9$, 
    which contradicts that $|\C_{\whG}(\whg)|\leq |\C_G(g)|=6$. 
    Therefore, $\whH$ is solvable. In fact, it is a $5$-group.

    Suppose that $H=K$.
    By the tables of $5$-modular characters of $S_5$, $S_6$ and $\Aut(A_6)$,
    $\olv$ acts on $\whH$ with non-trivial fixed points.
    Hence, there exists an element of order $15$ in $\whG$.
    Since $\whG$ is rational, there exists an element of order $4$ in $\C_{\whG}(\whv)$,
    which contradicts that $|\C_{\whG}(\whg)|=6$.
    Hence $K/H$ is not trivial.
    
    Let $H=H_0\leq H_1\leq \ldots \leq H_r=K$, 
    where $H_i/H_{i-1}$ is a chief factor of $G/H$, $1\leq i\leq r$.
    Each quotient $H_i/H_{i-1}$ is a $2$-group or a $3$-group.
    Suppose that $2\in \pi(K/T)$.
    Let $t$ be the smallest integer such that $H_t/H_{t-1}$ is a $2$-group.
    Thus $H_{t-1}/H$ is a $3$-group.
    Since $v\notin K$, $v'=vH_t/H_{t-1}$ is of order $3$.
    Let $Y\leq G$ such that $Y/H_{t-1}=\gen{H_t/H_{t-1},v'}$.
    Since $v'$ acts on $H_t/H_{t-1}$ without non-trivial fixed points,
    $Y/H_{t-1}$ is a Frobenius group.
    Let $P\in \Syl_3(Y/H)$ such that $vH/H\in P$.
    It is clear that $P=H_{t-1}/H \rtimes \gen{vH/H}$.
    By Lemma \ref{complement}, 
    $ H_{t-1}/H$ has a complement in $Y/H$ which is isomorphic to $Y/H_{t-1}$.
    Denote this group by $W/H$ and let $v^*\in W$ be the preimage of $v'$ in $W/H$.
    By the definition of $T$ and $H$, we have $\C_{\whK}(\whH)=\whH$.
    It follows from Lemma \ref{Frob} that there exists an element of order $15$ in $\whW$.
    Let $v^\dagger$ be the preimage of $v^*$ in $\whW$.
    Since $\whG$ is rational, there exists an element of order $4$ in $\C_{\whG}(v^\dagger)$.
    It follows that there exists an element of order $4$ in $\C_{G/H_{t-1}}(v')$,
    which contradicts that $|\C_{G/H_{t-1}}(gH_{t-1})|=6$.
    Hence $2\notin \pi(K/T)$ and so $K/H$ is a $3$-group.
    Since $\olu$ acts on $K/H_{r-1}$ without non-trivial fixed points, 
    from the tables of $3$-modular characters of $S_5$, $S_6$ and $\Aut(A_6)$, 
    it follows that $K/H_{r-1}$ is a one-dimensional irreducible $G/K$-module.
    By Lemma \ref{KS5} and Lemma \ref{real3auta6}, $G/H_{r-1}$ is not rational, a contradiction.
    Hence $5\notin \pi(K)$.

\medskip
{\sl (6) $3\notin \pi(K)$. Hence $K$ is a $2$-group.}  

    Suppose that $3\in \pi(K)$.
    Let $T$ be the largest normal subgroup of $G$ such that $|K/T|$ is divisible by $3$.
    Choose $H\leq G$ such that $H/T$ is a chief factor of $G$.
    Thus $H/T$ is a $3$-group and $K/H$ is a $2$-group.

    Suppose that $K=H$. In this case, $K/T$ is an irreducible $3$-module for $\overline G$.
    If $\overline G\cong S_5$, then Lemma \ref{KS5} implies that $G/T$ is not rational, a contradiction. 
    Hence $\overline G\cong S_6$ or $\Aut(A_6)$.
    By the Brauer character tables of these groups, if $\dim(K/T)>1$, then $uT$ has a non-trivial fixed point on $K/T$. 
    It follows that the $3$-part of $|\C_{G/T}(gT)|$ is at least $9$, a contradiction.
    Hence $\dim(K/T)=1$.
    However, Lemma \ref{real3auta6} shows that $G/T$ is not rational, also a contradiction. 
    Thus $K\neq H$, and so $K/H$ is non-trivial.

    Set $\whG=G/T$ and $\wtG=G/H$.
    Since every element of order $6$ in $S_6$ or $\Aut(A_6)$ lies in a subgroup isomorphic to $S_5$, 
    we can choose a subgroup $M \leq G$ such that $\olM \cong S_5$ containing $\olg$.
    It is clear that in $\wtM$, $|\C_{\wtM}(\wtg)|=6$.
    By Lemma \ref{2or16}, $\wtM\cong \wtK \rtimes \wtW$ with $\wtW\cong S_5$.
    We can choose the subgroup $W$ such that $\wtg$ lies in $\wtW$.
    Since $\whu$ acts on $\whH$ without non-trivial fixed points, 
    $\whH$ is a product of $1$-dimensional modules over $\wtW$.
    Therefore, $\whv$ must act on $\whH$ trivially.
    Let $L\leq G$ such that $\whL=\C_{\whG}(\whH)$.
    In particular, $\wtv\in \wtL$.
    By the definition of $T$, $\C_{\whK}(\whH)=\whH$.
    It follows that $\wtL\cap \wtK=1$.
    Since $\wtL$ and $\wtK$ are both normal in $\wtG$, 
    we have $\wtL$ and $\wtK$ centralize each other.
    It follows that $\wtK\leq \C_{\wtG}(\wtv)$, a contradiction.
    Hence (6) holds.

\medskip
{\sl (7) The group $G$ does not contain an element $h$ such that $|\C_G(h)|=5$. } 

    Suppose that this is false. 
    Since $|G|_5=5$ and $G$ is an $ah$-group,
    all $5$-elements in $G$ are conjugate.
    Hence for $t\geq 2$, $5t\notin \Omega = \{|\C_G(x)| \mid x\in G\}$.
    Note that $|G|_3\leq 9$.
    The sum of all conjugacy classes of $G$ is
    $$|G| < [2(1+\frac{1}{3} + \frac{1}{9})+\frac{1}{5} - 1 - \frac{1}{2} - \frac{1}{3} - \frac{1}{4} - \frac{1}{9}]|G| = \frac{161}{180} |G| < |G|, $$ 
    a contradiction. 
    Hence (7) holds.
    
\medskip
{\sl (8) Final contradiction.}

    Since the centralizers of elements of order $5$ in $S_5$, $S_6$ and $\Aut(A_6)$ contain no  elements of order $3$, 
    it follows that for any $x\in G$, if $5$ divides $|\C_G(x)|$,
    then $|\C_G(x)|$ must be of the form $5\times 2^k$ with $k>1$.
    Estimating the sum of the conjugacy class sizes, we have 
    $$ |G| < [2(1+\dfrac{1}{3}+\dfrac{1}{9}) + (\dfrac{1}{10} + \dfrac{1}{20} + \cdots )
    - (1 + \dfrac{1}{2} +\dfrac{1}{3}+\dfrac{1}{4} +\dfrac{1}{9})]|G|<\dfrac{161}{180}|G|<|G|. $$
    This contradiction completes the proof.
\end{proof}

\begin{lem}
    For any element $g \in G$, $|\C_G(g)| \neq 5$.
\end{lem}

\begin{proof}
    Suppose this is false and there exists $g\in G$ such that $|\C_G(g)|=5$.
    By Lemma \ref{Sylp}, $\gen{g}$ is a Sylow $5$-subgroup of $G$.
    Since $G$ is an $ah$-group, all $5$-elements in $G$ are conjugate.
    Hence, for $t\geq 2$, $5t\notin \Omega = \{|\C_G(x)| \mid x\in G\}$.
    Estimating the sum of the conjugacy class sizes, we have 
    $$ |G| < [(1+\dfrac{1}{2} + \cdots)\times (1 + \dfrac{1}{3} + \cdots ) + \dfrac{1}{5}
    - (1 + \dfrac{1}{2} +\dfrac{1}{3}+\dfrac{1}{4} +\dfrac{1}{6} +\dfrac{1}{9})]|G|
    <\dfrac{151}{180}|G|
    <|G|. $$
    This contradiction completes the proof.
\end{proof}

\begin{lem}
    For any element $g \in G$, $|\C_G(g)| \neq 8$.
\end{lem}

\begin{proof}
    Suppose this lemma is false and there exists an element $g\in G$ such that $|\C_G(g)|=8$.

{\sl (1) $|\olg|\neq 8$. } 

    Suppose that $|\olg|=8$. 
    By Lemma \ref{trivial}, $|K|$ is not divided by $2$.
    It follows that $K$ is solvable.

    Suppose that $R_i^{\olg} \neq R_i$.
    Let $\calO = \{R_i^{\olg^n}\mid n\in \bbN\}$ and $U=\bigoplus_{X\in \calO} X$.
    Hence $|\calO|\in \{1,2,4,8\}$.
    
    Assume $|\calO|=8$. Therefore $\C_U(\olg)$ contains a subgroup isomorphic to $R_i$, a contradiction.
    
    Assume $|\calO|=4$. Therefore $\C_U(\olg)$ contains a subgroup isomorphic to $\C_{R_i}(x)$ 
    where $x$ is an element of order $2$ in $\Aut(R_i)$,
    viewing $R_i$ as a subgroup of $\Aut(R_i)$.
    In this case, $\olg$ and $\olg^2$ are both not contained in $R$.
    Hence $\C_R(\olg)=\C_{\olG}(\olg)\cap R = \{1\}$ or $\{1,\olg^4\}$.
    However, by \cite[Code 1]{Git},
    $|\C_{R_i}(x)|\geq 4$ regardless of whether $R_i\cong A_5, A_6$ or $\PSp_4(3)$, a contradiction.
    
    Assume $|\calO|=2$. Therefore $|\C_{R}(\olg)|\leq 4$.
    By \cite[Code 1]{Git}, 
    every automorphism of $\PSp_4(3)$ of order $2$ or $4$
    has a fixed-point subgroup in $\PSp_4(3)$ of order at least $8$, a contradiction.
    Hence $R_i$ must be isomorphic to either $A_5$ or $A_6$.
    In these cases, the order of the centralizer of $\olg$ in $U\rtimes \gen{\olg}$ is at least $8$, 
    which forces $R=U$.
    Note that $\olG$ is a group such that $U\leq \olG \lesssim \Aut(U)$.
    In the case where $R\cong A_6\times A_6$,
    it can be verified by \cite[Code 7]{Git} that for all possible $\olG$, 
    there exists no self-centralizing element of order $8$ in $\olG$.
    Hence $R\cong A_5\times A_5$.
    Furthermore, it can be verified by \cite[Code 3]{Git} that $\olG\cong \Aut(A_5\times A_5)$.

    Choose a normal subgroup $T\leq K$ of $G$ such that $K/T$ is a chief factor of $G$.
    Then $K/T$ can be viewed as an irreducible $3$- or $5$-module for $\olG$.
    By \cite[Code 8]{Git}, $G/T$ cannot be both rational and satisfy $|\C_{G/T}(gT)|=8$.
    This contradiction shows that $R\not\cong A_5\times A_5$. 
    Hence $R_i^{\olg}=R_i$ for all $i=1,\ldots,k$.

    It is easy to verify that $|\C_{R_i}(\olg)|>1$ for $i=1,\ldots,k$. 
    Since $\C_R(\olg)\leq \gen{g}$ is a cyclic group, we have $R=R_1$.
    Since $\overline{G}\leq \Aut(R)$ is a rational group containing a self-centralizing element of order $8$,
    $\overline{G}$ must be isomorphic to either $\Aut(A_6)$ or $\Aut(\PSp_4(3))$. 

    Suppose that $\olG \cong \Aut(A_6)$.
    Since $\Aut(A_6)$ contains a self-centralizing element of order $6$,
    denote by $Y$ the complete preimage of its conjugacy class in $G$.
    By Lemma \ref{Orbit}, we have
    $$\sum_{i=1}^{l} \dfrac{1}{|\C_G(x_i)|} = \dfrac{1}{6},$$
    where $x_1, \ldots, x_l$ are representatives of the distinct $G$-conjugacy classes in $Y$.
    For $i=1,\ldots,l$, since $\overline{x}_i$ has order $6$, we have
    $6$ divides $|x_i|$, and thus $6$ divides $|\C_G(x_i)|$.
    Since $K$ is a $\{3,5\}$-group, $|\C_G(x_i)|_2 = 2$ for $i=1,\ldots,l$.
    Moreover, since $|\C_G(x_i)| \neq 6$ for $i=1,\ldots,l$, we have
    $$\sum_{i=1}^{l}\dfrac{1}{|\C_G(x_i)|} \leq \dfrac{1}{6} \left[ \left(1 + \dfrac{1}{3} + \dfrac{1}{9} + \cdots\right)\left(1 + \dfrac{1}{5} + \dfrac{1}{25} + \cdots\right) - 1 \right] = \dfrac{1}{6} \left(\dfrac{3}{2} \cdot \dfrac{5}{4} - 1\right) < \dfrac{1}{6}.$$
    This contradiction shows that $\olG \not\cong \Aut(A_6)$.

    Similarly, since there exists a self-centralizing element of order $9$ in $\PSp_4(3)$ and $|\C_G(x)| \neq 9$ for any element $x \in G$, 
    we can also get that $\olG \not\cong \PSp_4(3)$.
    Hence $|\olg| \neq 8$.

\medskip
{\sl (2) $|\olg|\neq 4$. } 

    Suppose that $|\olg|=4$ and we get a contradiction by the following steps.

{\sl (2.1) $\overline G$ is isomorphic to one of the following groups:
$S_5$, $S_6$, and $\Aut(A_6)$.} 
    
    Let $\calO = \{R_1^{\olg^n}\mid n\in \bbN\}$ and $M=\bigoplus_{X\in \calO} X$.
    Suppose that $|\calO|>1$.
    Since $\olg$ has order $4$ and $R_1^{\olg} \neq R_1$, it follows that $|\calO|\in \{2,4\}$.
    Assume that $|\calO|=4$.
    Therefore $\C_M(\olg)$ contains a subgroup isomorphic to $R_1$, a contradiction.
    If $|\calO|=2$,
    then $\C_M(\olg)$ contains a subgroup isomorphic to $\C_{R_1}(\olg^2)$.
    Note that $|\C_{\olG}(\olg)|\leq |\C_G(g)|=8$.
    By \cite[Code 1]{Git}, we have $R_1$ must be isomorphic to either $A_5$ or $A_6$.
    Furthermore, this forces $k=2$ and so $R=M$.
    However, by \cite[Code 3, Code 7]{Git}, 
    among the rational groups $\olG$ satisfying $R\leq \olG\leq \Aut(R)$, 
    with $R\cong A_5\times A_5$ or $R\cong A_6\times A_6$, 
    the only case in which $|\C_{\olG}(\olg)|=8$ is $\olG\cong A_5\rtimes S_5$.
    This contradicts the definition of $K$.
    Therefore, we must have $R_i^{\olg}=R_i$ for $i=1,\ldots,k$.
    
    If $R_1\cong \PSp_4(3)$, by \cite[Code 1]{Git},
    $\olg$ acts on $R_1$ as an inner automorphism,
    which forces $R=R_1$ and $\olG\cong \Aut(\PSp_4(3))$.
    However, the order of the centralizer of any element of order $4$ in $\Aut(\PSp_4(3))$ is greater than $8$.
    This contradiction implies that $R_i$ is isomorphic to either $A_5$ or $A_6$. 
    If $\olg$ acts on $R_1$ as an inner automorphism, 
    then by \cite[Code 1]{Git}, $R_1\cong A_6$ and $k=1$.
	    Hence we may suppose that $\olg$ acts on $R_1$ as an outer automorphism.
	    Then $|\C_{R_i}(\olg)|\geq 2$ for $1\leq i\leq k$.
    If $k\geq 3$, then $|\C_R(\olg)|\geq 8$.
	    Since $\olg \notin R$, we have $|\C_{\olG}(\olg)|\geq 16$, a contradiction.
    Thus, $k\leq 2$.
    Similarly, by \cite[Code 3, Code 7]{Git}, we get $k\neq 2$.
    Therefore, $R\cong A_5$ or $A_6$ and (2.1) holds.

\medskip
{\sl (2.2) $K$ is solvable. } 
    
    Let $B$ be the solvable radical of $G$, 
    and set $\wtG=G/B$. 
    Choose a minimal normal subgroup $W = W_1 \times \dots \times W_m$ in $\wtG$, 
    where $W_1, \dots, W_m$ are isomorphic non-abelian simple groups. 
    The element $\wtg$ acts on $W$. 
    Since $|\olg|=4$ and $|\C_G(g)|=8$, we have $|\C_{K}(g)|\leq 2$.
    Hence $|\C_W(\wtg)|\leq 2$.
    Therefore, by \cite[Code 1]{Git}, $m=1$ and $W$ is isomorphic to $A_5$ or $A_6$.
    Let $E$ be the normal subgroup of $G$ such that $E/W\cong \wtK\cap \C_{\wtG}(W)$.
    We have 
    $$G/E\cong \wtG/\wtE \lesssim \wtG/\wtK \times \wtG/\C_{\wtG}(W)
        \lesssim \olG \times \Aut(W).$$
    It is clear that $G/E$ is rational and contains an element whose centralizer has order $4$ or $8$.
    By \cite[Code 3, Code 7, Code 9]{Git}, $G/E\cong (A_5\times A_5).C_2$.
    If $E$ is not solvable, then by the same argument above, 
    it is easy to get that $|\C_G(g)|>8$, a contradiction.
    Hence $E=B$ and so $G/E=\wtG$.

    Let $T\leq B$ such that $B/T$ is a chief factor of $G$.
    Then $B/T$ can be viewed as an irreducible $2$-, $3$- or $5$- module of $\wtG$.
    By \cite[Code 10]{Git}, $G/T$ cannot both be rational and satisfy that $|\C_{G/T}(gT)|=8$.
    Hence $K$ must be solvable.

\medskip
{\sl (2.3) $ \olG \not\cong S_6$ or $\Aut(A_6)$. Hence $\olG\cong S_5$. } 

    By \cite[Code 11]{Git}, it can be verified that 
    when $\overline{G} \cong S_6$ or $\Aut(A_6)$, 
    no extension of an irreducible $2$-, $3$-, or $5$-module by $\overline{G}$
    can both be a rational group and contain an element with a centralizer of order $8$.
    Hence (2.3) holds.

\medskip
{\sl (2.4) Let $N\teq G$ such that $N\leq K$ and $K/N$ is a 2-group.
    We have that $G/N$ is isomorphic to $C_2\times S_5$.} 

    By Lemma \ref{KS5}, $K/N>1$.
    Assume that $G/N = V_2.V_1.\olG$, 
    where $V_1$ and $V_2$ are irreducible 2-modules for $\olG$ and $V_1.\olG$, respectively.
    By Lemma \ref{KS5},
    $V_1.\olG$ is isomorphic to one of the following groups: 
    $C_2\times S_5$, $[1920, 240993]$ and $[1920, 240996]$.
    
    If $V_1.\olG\cong C_2\times S_5$, then the irreducible 2-modules for
    $V_1.\olG$ have orders $2$ and $16$.
    By \cite[Code 12]{Git}, if $|V_2|=2$, the only rational case is
    $C_2\times C_2\times S_5$, which has no element with centralizer of order $8$;
    if $|V_2|=16$, all possible extensions again have no such element.
    Thus this case is impossible.

    Hence $V_1.\olG$ is isomorphic to $[1920,240993]$ or $[1920,240996]$.
    For both groups, the irreducible 2-modules have orders $2$ and $16$.
    By \cite[Code 13]{Git}, if $|V_2|=2$, 
    then the only rational possibility for $V_2.V_1.\olG$ is
    $C_2\times (V_1.\olG)$, which contains no element with centralizer of order $8$.
    Thus $|V_2|=16$.
    In this case, $V_2.V_1.A_5$ is perfect, and $G/N$ is isomorphic to a subgroup of $\Aut(V_2.V_1.A_5)$.
    By \cite[Code 13]{Git}, no such group $G/N$ is both rational 
    and contains an element with centralizer of order $8$, a contradiction.

    It remains to exclude the cases where $G/N\cong [1920,240993]$ or $[1920,240996]$.
    Choose $J\leq N$ such that $J\teq G$ and $N/J$ is a chief factor of $G$.
    By the previous argument, $N/J$ is not a 2-group, and hence it is a $3$- or $5$-group.
    However, \cite[Code 14]{Git} shows that 
    no extension of either $[1920,240993]$ or $[1920,240996]$ by an irreducible $3$- or $5$-module
    can be both rational and contain an element with centralizer of order $8$.
    Thus $G/N\cong C_2\times S_5$, and (2.4) holds.
    
\medskip
{\sl (2.5) $G\cong (C_3^n\rtimes C_2)\times S_5$. Hence it is not an $ah$-group.}

    Let $D \leq N$ such that $D \te G$ and $N/D$ is a chief factor of $G$. 
    By (2.4), $N/D$ can be viewed as an irreducible $3$- or $5$-module for $G/N \cong C_2 \times S_5$.
    Suppose that $N/D$ is a $5$-group.
    By \cite[Code 15]{Git}, it can be verified that 
    $G/D$ cannot both be a rational group and contain an element whose centralizer has order $8$.
    Hence $N/D$ must be a 3-group.

    Let $L\leq N$ be the smallest normal subgroup of $G$ such that $N/L$ is a $3$-group.
    By Lemma \ref{C2S5}, $G/L\cong (C_3^n\rtimes C_2)\times S_5$, where $n\geq 1$.
    By Lemma \ref{S3S5}, $L=1$.
    Thus, $G\cong (C_3^n\rtimes C_2)\times S_5$.
    It is obvious that $G$ is not an $ah$-group.
    Hence (2) holds.

\medskip    
{\sl (3) $|\olg|\neq 2$. }

    Suppose that $|\olg|=2$.
    It is easy to verify that $\olG\cong S_5$.
    In particular, $\olg$ has cycle type $(1,2)(3,4)$ in $S_5$,
    $|\C_{\olG}(\olg)|=8$, and $\olg$ lies in the normal subgroup $A\cong A_5$ of $\olG$.
    Let $H<K$ be a normal subgroup of $G$ such that $K/H$ is a chief factor of $G$,
    and set $\widetilde G=G/H$.

    Suppose that $\widetilde K$ is not solvable.
    We have $\widetilde K=L_1\times\cdots\times L_n$,
    where $L_1,\ldots,L_n$ are isomorphic non-abelian simple groups,
    each isomorphic to one of $A_5$, $A_6$, or $\PSp_4(3)$.
    Since $\widetilde g\notin \widetilde K$ and
    $|\C_{\widetilde G}(\widetilde g)|\leq |\C_G(g)|=8$,
    we have $|\C_{\widetilde K}(\widetilde g)|\leq 4$.

    If $L_1^{\widetilde g}\neq L_1$, then we may assume
    $L_1^{\widetilde g}=L_2$.
    In this case $\widetilde g^2\in \widetilde K$.
    By Lemma \ref{Centre2}, $\C_{L_1L_2}(\widetilde g)$ contains a subgroup
    isomorphic to $\C_{L_1}(\widetilde g^2)$,
    whose order is at least $4$ by \cite[Code 1]{Git}.
    Hence $n=2$.
    Since $\widetilde G/\C_{\widetilde G}(\widetilde K)\lesssim \Aut(\widetilde K)$
    and $\olg$ lies in the normal subgroup $A\cong A_5$ of $\olG$,
    we have $\widetilde g\in \C_{\widetilde G}(\widetilde K)\widetilde K$.
    This contradicts $L_1^{\widetilde g}\neq L_1$.

    Therefore $L_i^{\widetilde g}=L_i$ for $i=1,\ldots,n$.
    By \cite[Code 1]{Git}, $|\C_{L_i}(\widetilde g)|\geq 2$ for each $i$,
    and so $n\leq 2$.
    It follows that $\widetilde G\cong (P\times Q)\rtimes C_2$,
    where $Q\cong A_5$ and $\widetilde g\in P\times Q$.
    It is easy to verify that $|\C_{P\times Q}(\widetilde g)|>8$,
    a contradiction.
    Thus $\widetilde K=K/H$ is solvable.

    By Lemma \ref{KS5}, $\widetilde K$ is a $2$-group and $\widetilde G$ is
    isomorphic to one of $C_2\times S_5$, $[1920,240993]$ or $[1920,240996]$.
    However, \cite[Code 16]{Git} shows that, in each of these three groups,
    every element whose image in $S_5$ has cycle type $(1,2)(3,4)$ has centralizer
    of order greater than $8$.
    This contradicts $|\C_{G/H}(gH)|\leq |\C_G(g)|=8$.
\end{proof}

\begin{lem}
    $G$ does not contain a pair of elements $g_1$ and $g_2$ such that $|\C_G(g_1)|=10$ and $|\C_G(g_2)|=15$.
\end{lem}

\begin{proof}
    Suppose that $\{10,15\}\subset \N(G)$. 
    Let $g_1,g_2\in G$ be such that $|\C_G(g_1)|=10$, $|\C_G(g_2)|=15$. 
    We can assume that $|g_1|=10$ and $|g_2|=15$. 
    Let $g_1=h_1 f_1$, $g_2=h_2 f_2$, where $h_1=g_1^5$, $f_1 = g_1^2$, $h_2 = g_2 ^5$, $f_2= g_2^3$.
    It follows from Lemma \ref{Centre} that $\overline g_1$, $\overline g_2$ are nontrivial. 

{\sl  (1) $k=1$ and $\olG\in\{S_5, S_6, \Aut(A_6), \Aut(\PSp_4(3))\}$. } 
    
    Suppose that $k>2$. 
    It is easy to show that 
    in this case,  for any element $x\in \overline{G}$ of order $2, 5$ or $10$,
    $|\C_{\overline{G}}(x)|>10$.
    In particular, $|\C_{\overline G}(\overline g_1)|>10$, a contradiction.

    Suppose $k=2$.
    Assume that $R_1^{\overline g_1}=R_1$. 
    If $|\overline g_1|=5$ or $10$, then $|\C_{R_i}(\overline g_1)|\geq 5$ for $i=1,2$ 
    and so $|\C_{\olG}(\overline g_1)|\geq 25$, a contradiction.
    If $\overline g_1 =  \overline h_1$, 
    by \cite[Code 1]{Git}, we have $|\C_{R_i}(\overline g_1)|\geq 4$ for $i=1,2$,
    which implies that $|\C_R(\overline g_1)|\geq 16$, also a contradiction.
    Hence $R_1^{\overline g_1}=R_2$.
    If $R\cong \PSp_4(3)\times \PSp_4(3)$, then $\olg_2\in R$.
    It is clear that $|\C_{\olG}(\olg)|>15$, a contradiction.
    Hence $R\cong A_5\times A_5$ or $A_6\times A_6$.
    Since $\olG$ is rational with $|\C_{\olG}(\overline g_1)|\leq 10$ and $|\C_{\olG}(\overline g_2)|\leq 15$,
    by \cite[Code 3, Code 7]{Git}, we get $\overline G\cong (A_5\times A_5)\rtimes (C_2\times C_2)$.
    Moreover, we have $|\C_{\olG}(\overline g_1)|=10$, $|\C_{\overline{G}}(\overline g_2)|=15$.

    Suppose that $5$ divides $|K|$. 
    Let $T\leq K$ be the largest normal subgroup of $G$ such that $|K/T|$ is divisible by $5$. 
    Set $\wtG=G/T$ and let $A\leq K$ be such that $\wtA$ is the minimal normal subgroup of $\wtG$. 
    Thus $|\wtA|$ is divisible by $5$ and $K/A$ is a $\{2,3\}$-group. 
    Assume that $\wtA$ is non-abelian. Thus $\wtA$ is a direct product of non-abelian simple groups. 
    By Lemma \ref{Centre2}, we have $|\C_{\widetilde{A}}(\widetilde g_1)|>1$.
    Since $\wth_1, \wtf_1\notin \wtA$, 
    $|\C_{\widetilde{G}}(\widetilde g_1)|> 10$, a contradiction. 
    Therefore, $\wtA$ is an elementary abelian $5$-group. 
    Set $\whG=G/A$. 
    
    Suppose that $\C_{\whK}(\widehat f_1)>1$.
    Since $|\C_{\whG}(\widehat f_1)|_2=2$ and $\widehat h_1 \notin \whK$, 
    we have that $\C_{\whK}(\widehat f_1)$ is a $3$-group.
    It is clear that $\C_{\C_{\whK}(\widehat f_1)}(\widehat h_1)=1$; 
    otherwise $|\C_{\whG}(\widehat g_1)|>10$, a contradiction.
    Since $\widehat h_1$ normalizes $\C_{\whK}(\widehat f_1)$, 
   we have $\C_{\whK}(\widehat f_1)\gen{\widehat h_1}$ is a Frobenius group.
    From the definition of $A$, we have $\C_{\wtK}(\wtA)=\wtA$.
    By Lemma \ref{Frob}, $\widehat h_1$ acts non-trivially on $\wtA$, a contradiction.
    Hence $\C_{\whK}(\widehat f_1)=1$, which implies that $\whK$ is nilpotent.
    Suppose that $|\whK|_3 >1$ and let $U\leq \whK$ be a Sylow $3$-subgroup of $\whK$.
    If $\overline h_1$ acts on $U$ trivially, then $\gen{\overline h_1^{\olG}}$ acts on $U$ trivially.
    Note that $\overline f_1\in \gen{\overline h_1^{\olG}}$.
    We have $\overline f_1$ acts on $U$ trivially, contradicting that $\C_{\whK}(\widehat f_1)=1$.
    Hence $\overline h_1$ acts on $U$ non-trivially.
    By Lemma \ref{Mazurov06}, we have $\overline h_1$ acts on $\wtA$ with non-trivial fixed points, a contradiction.
    Hence $\whK$ is a $2$-group.

    Let $D$ be a minimal normal subgroup of $\whG$ in $\whK$.
    Hence $D$ is an elementary abelian $2$-group.
    By Lemma \ref{Go}, we have $D=\C_D(\gen{\whh_2})\times [D, \gen{\whh_2}]$.
    If $D=\C_D(\gen{\whh_2})$, then $\gen{{\whh_2}^{\whG}}$ also centralizes $D$,
    which implies that $\whf_2$ centralizes $D$, a contradiction.
    Hence $[D,\gen{\whh_2}]>1$ 
    and $[D,\gen{\whh_2}]\gen{\whh_2}$ is a Frobenius group.
    By Lemma \ref{Frob}, $\whh_2$ has non-trivial fixed points in $\wtA$, a contradiction.
    Hence $5\notin \pi(K)$.

    Note that in this case, 
    the centralizer of any $5$-element in $\olG$ has order greater than $25$. 
    Hence there exists no element in $G$ whose centralizer has order exactly $25$.
    We have 
    $$ |G| \leq [2\times \frac{3}{2}\times (1+\frac{1}{5}+\frac{1}{25}) 
    -  (1 + \frac{1}{2} + \frac{1}{3} + \frac{1}{4} + \frac{1}{5} + \frac{1}{6}+ \frac{1}{8} + \frac{1}{9} + \frac{1}{25})]|G|
    = \frac{1789}{1800} |G|< |G|,$$ 
    a contradiction.
    Hence $k=1$ and $(1)$ holds.

{\sl  (2) $\overline {g}_2=\overline f_2$. } 

    Since there exists no element of order $15$ in $\olG$,
    we have $\olg_2 = \olh_2$ or $\olf_2$.
    Suppose that $\olg_2 = \olh_2$.
    Since $|\C_{\olG}(\olg_2)|\leq |\C_G(g_2)|=15$,
    we have $\olG \cong S_5$.
    
    Let $P=\gen{f_2}$.
    Suppose that $P$ lies in an abelian composition factor of $G$.
    Thus $P$ also lies in an abelian composition factor of $\C_G(h_2)$.
    Since $|\C_G(h_2)|_5=5$, $\C_G(h_2)$ must be solvable.
    Let $M=\rmO_{5'}(\C_G(h_2))$. 
    We have that $PM/M$ is self-centralizing in $\C_G(h_2)/M$.
    By Lemma \ref{CGU}, $\C_G(h_2)$ is a $3'$-qr group. 
    It follows that $(\C_G(h_2)/M)/(PM/M)\cong \Aut(P)\cong C_4$.
    Hence $\C_G(h_2)/M\cong C_5\rtimes C_4$, 
    which is not a $3'$-qr group, a contradiction.
    Thus, $P$ lies in a non-abelian chief factor of $G$,
    denoted by $W$.
    
    We show that $W$ must be the factor $R$ of $\olG$.
    Suppose the contrary. Thus there exist normal subgroups $Y<X\leq K$ of $G$
    such that $W=X/Y$ is a non-abelian chief factor of $G$.
    Let $L/Y=\C_{G/Y}(W)\cap K/Y$.
    Set $\whG=G/L$ and $\whX=XL/L$.
    Then $\whX\cong W$.
    Let $\whX=U_1\times \cdots \times U_m$,
    where $U_1,\ldots, U_m$ are isomorphic to one of $A_5$, $A_6$ or $\PSp_4(3)$.
    Moreover,
    $$ \whG \cong (G/Y)/(\C_{G/Y}(W)\cap K/Y)
        \lesssim \Aut(W)\times \olG
        \cong ((\Aut(U_1))^m\rtimes S_m)\times S_5. $$

    Suppose that $R$ acts non-trivially on $W$, that is, $\whG\lesssim \Aut(W)$.
    Thus $m\geq 5$.
    It is clear that $\whf_2\neq 1$ and $\whf_2\in \whX$.
    Let $\whf_2=u_1\cdots u_m$, where $u_i\in U_i$ for $i=1,\ldots,m$.
    If there exists some $i$ such that $u_i=1$,
    then $U_i\leq \C_{\whX}(\whf_2)$,
    which implies that $|\C_{\whG}(\whg_2)|_3\geq 9$,
    a contradiction.
    Thus $u_i\neq 1$ for each $i$.
    Since $\whh_2$ centralizes $\whf_2$, 
    the action of $\whh_2$ on any $U_i$ stabilized by $\whh_2$ must be trivial.
    If $U_i^{\whh_2} \neq U_i$, then by Lemma \ref{Centre2},
    the centralizer of $\whh_2$ in $U_iU_i^{\whh_2}U_i^{{\whh_2}^2}$ contains a subgroup isomorphic to $U_i$.
    Hence $|\C_{\whX}(\whh_2)|_5\geq 25$ and so $|\C_{\whG}(\whg_2)|_5\geq 25$, a contradiction.
    Thus $R$ acts trivially on $\whX$.

    Since $|\C_{\whG}(\whh_2)|_5\leq 5$ and $|\C_{\whG}(\whf_2)|_3=3$, 
    it is easy to get that $m=1$
    and $\whG \cong (A_l\times A_5).C_2$, where $l\in\{5,6\}$; otherwise $|\C_F(\whg_2)|>15$, a contradiction. 
    It can be verified that for any $x\in F$, if $|x|=2, 5$ or $10$, then $|\C_F(x)|>10$.
    Hence $|\C_{\whG}(\whg_1)|>10$, a contradiction.
    Thus, $W=R$ and so $|\overline g_2|_5=5$.
    Since there exists no element of order $15$ in $\olG$, we have $|\overline g_2|=5$,
    a contradiction with $\olg_2=\olh_2$.
    Therefore $\overline {g}_2=\overline f_2$.

{\sl  (3) $\overline g_1\neq \overline h_1$. } 

    Assume that $\olg_1 = \olh_1$. 
    Since $|\C_{\olG}(\olg_1)|\leq |\C_G(g_1)|=10$,
    we have $\olG\cong S_5$ and $\olg_1\in R$. 
    Let $T<K$ be a normal subgroup of $G$ such that $K/T$ is a chief factor of $G$.
    Set $\wtG=G/T$.
    Suppose that $\wtK=K/T$ is not abelian.
    We have $\wtK\cong U_1\times \cdots \times U_m$,
    where $U_1, \ldots, U_m$ are isomorphic to one of $A_5$, $A_6$ and $\PSp_4(3)$.

    We have $\wtG/\C_{\wtG}(\wtK)\lesssim \Aut(\wtK) \cong (\Aut(U_1))^m \rtimes S_m$.
    Suppose that $R$ acts on $\wtK$ non-trivially.
    We may choose $U_1$ such that $U_1^{\wth_1}\neq U_1$.
    By Lemma \ref{Centre2}, 
    the centralizer of $\wth_1$ in $U_1U_1^{\wth_1}$ contains a subgroup isomorphic to $U_1$.
    Since $\olh_1$ has cycle type $(1,2)(3,4)$ in $\olG\cong S_5$,
    $\wth_1$ acts on the direct factors of $\wtK$ as a product of two disjoint transpositions.
    Hence $|\C_{\wtK}(\wth_1)|_5\geq 25$,
    which implies that $|\C_{\wtG}(\wtg_1)|>10$.
    This contradiction shows that $R$ acts on $\wtK$ trivially.
    If $\wtf_1=1$, 
    then $|\C_{\wtG}(\wtg_1)|=|\C_{\wtG}(\wth_1)|\geq |\wtK|>10$, a contradiction.
    Thus $\wtf_1 \neq 1$ and $\wtf_1\in \wtK$.
    It follows that $|\C_{\wtG}(\wtg_1)|\geq |\gen{\wtf_1}||\C_R(\olh_1)|=20$, also a contradiction.
    Hence $\wtK$ is abelian.
    
    By Lemma \ref{KS5},
    $\wtG$ is isomorphic to 
    one of $C_2\times S_5$, $[1920, 240993]$ or $[1920, 240996]$.
    Since $\wtg_1$ is an element of order $2$, it is easy to check in each of these
    three groups that $|\C_{\wtG}(\wtg_1)|\geq 16$, a contradiction.   
    Hence (3) holds.

{\sl  (4) $5$ does not divide $|K|$. } 

    Assume that $5$ divides $|K|$. 
    Let $N\teq K$ be the subgroup generated by all Sylow $5$-subgroups of $K$.
    Choose $T\teq G$ with $T\leq N$ such that $N/T$ is a chief factor of $G$.
    Set $\wtG=G/N$ and $\whG=G/T$.
    The definition of $N$ implies that $\whN=N/T$ is a direct product of simple groups and $5$ divides $|\whN|$. 
    If $\whN$ is non-abelian, we can get a contradiction as in step (2).
    Hence $\whN$ is abelian.
    
    Assume that $|\wtg_i|=5$ for some $i\in\{1,2\}$. 
    We have $|\whg_i|=5$. 
    However, $|\whG|_5\geq 25$ and so $|\C_{\whG}(\whg_i)|_5\geq 25$, a contradiction. 
    Hence $|\wtg_i|\neq 5$ for $i=1,2$.
    By steps (2) and (3), we have $|\wtg_i|=|g_i|$ for $i=1,2$.

    Since $|\wtG|_5=5$, all elements of order $5$ are conjugate in $\wtG$.
    Thus, without loss of generality, we can assume that $\wtf_1=\wtf_2$.
    We have $|\C_{\wtG}(\wtf_1)|_2=2$ and $|\C_{\wtG}(\wtf_1)|_3=3$.
    It is clear that $\wth_1$ and $\wth_2$ do not centralize each other.
    Therefore, $\C_{\wtG}(\wtf_1)=\gen{\wtf_1, \wth_1, \wth_2}\cong C_5\times S_3$. 
    Thus $\gen{\wth_1, \wth_2}\cong S_3$ is a Frobenius group.
    If $\C_{\whN}(\whh_2)>1$, then $|\C_{\whG}(\whh_2)|_5\geq 25$,
    which implies that $|\C_{\whG}(\whg_2)|_5\geq 25$, a contradiction.
    Hence $\C_{\whN}(\whh_2) = 1$.
    By Lemma \ref{Frob}, $\whh_1$ has non-trivial fixed points on $\whN$.
    It follows that $|\C_{\whG}(\whh_1)|_5 \geq 25$, also a contradiction.
    Therefore $5$ does not divide $|K|$. 

{\sl  (5) Final contradiction. } 

  Since $|G|_5=5$, we have 
  $$ |G| < [2\times \frac{3}{2}\times (1+\frac{1}{5}) 
    -  (1 + \frac{1}{2} + \frac{1}{3} + \frac{1}{4} + \frac{1}{5} + \frac{1}{6}+ \frac{1}{8} + \frac{1}{9})]|G|
    = \frac{329}{360} |G|< |G| .$$
This contradiction completes the proof.
\end{proof}

\begin{proof}[Proof of Main Theorem]
    If $G$ is solvable, then $G\cong S_3$ by the known result for solvable groups.
    Hence we may assume that $G$ is not solvable.
    Let $\Omega=\{|\C_G(x)|\mid x\in G\}$.
    Since $G$ is an $ah$-group, the elements of $\Omega$ correspond bijectively to the conjugacy classes of $G$.
    Therefore
    $$
        |G|=\sum_{m\in \Omega}\frac{|G|}{m}.
    $$
    By the lemmas above,
    $1,2,3,4,5,6,8,9\notin \Omega$.
    Moreover, $10$ and $15$ cannot both belong to $\Omega$.
    Estimating the sum of the conjugacy class sizes, we have
    $$
        |G| <
        \left[
        2\times \frac{3}{2}\times \frac{5}{4}
        -
        \left(
        1+\frac{1}{2}+\frac{1}{3}+\frac{1}{4}+\frac{1}{5}
        +\frac{1}{6}+\frac{1}{8}+\frac{1}{9}+\frac{1}{15}
        \right)
        \right]|G|
        =
        \frac{359}{360}|G|<|G|,
    $$
    a contradiction.
    This completes the proof.
\end{proof}

\end{document}